\documentclass{amsart}
\usepackage{amsmath, amsfonts, amsthm, latexsym, amssymb, hyperref, xcolor, multicol, dsfont, enumerate, bbold, enumitem}
\usepackage{geometry}
\usepackage[utf8]{inputenc}
\usepackage[T1]{fontenc}
\newcommand{\supp}{{\rm supp}  }

\def\R{{\mathbb R}}
\def\A{{\mathbb A}}

\def\M{{\mathbb M}}
\def\Z{{\mathbb Z}}
\def\D{{\mathbb D}}
\def\Dq{{\mathbb D}_q}
\def\H{{\mathbb H}}
\def\B{{\mathbb B}}
\def\Bz{{\mathbb B}(z)}
\def\Sz{{\mathbb S}(z)}
\def\Sb{\mathbb S}
\def\F{{\mathbb F}}
\def\X{{\mathbb X}}

\def\<{\langle}
\def\>{\rangle}
\def\P{\mathbb P}
\def\Q{\mathbb Q}
\def\E{\mathbb E}
\def\T{\mathbb T}
\def\0{\mathbb 0}
\def\W{\mathbb W}

\def\1{\underline 1}

\def\Y{\mathbb Y}
\def\var{\mathbb{V}\mathrm{ar}}

\newcommand{\bel}{\begin{equation}\label}
\newcommand{\ee}{\end{equation}}
      \newtheorem{theorem}{Theorem}[section]
       
       \newtheorem{corollary}[theorem]{Corollary}
       \newtheorem{lemma}[theorem]{Lemma}
       \newtheorem{remark}[theorem]{Remark}
\theoremstyle{definition}
\newtheorem{definition}{Definition}[section]

\newtheorem{example}{Example}[section]

\title{Infinitesimal generators for a family of polynomial processes \\ - an algebraic approach}
\author{Jacek Weso\l{}owski, Agnieszka Zięba}

\begin{document}
\maketitle
\section*{Abstract}
Quadratic harnesses are time-inhomogeneous Markov polynomial processes with linear conditional expectations and quadratic conditional variances with respect to the past-future filtrations. Typically they are determined by five numerical constants $\eta$, $\theta$, $\tau$, $\sigma$ and $q$ hidden in the form of conditional variances. In this paper we derive infinitesimal generators of such processes in the case $\sigma=0$, extending previously known results. The infinitesimal generators are identified through a solution of a $q$-commutation equation in the algebra $\mathcal Q$ of infinite sequences of polynomials in one variable. The solution is a special element in $\mathcal Q$, whose coordinates satisfy a three-term recurrence and thus define a system of orthogonal polynomials. It turns out that the respective orthogonality measure $\nu_{x,t}$ uniquely determines the infinitesimal generator (acting on polynomials or bounded functions with bounded continuous second derivative) as an integro-differential operator with the explicit kernel, where the integration is with respect to the measure $\nu_{x,t}$.\\

\underline{Key words}:  polynomial processes, quadratic harnesses, infinitesimal generators, orthogonal polynomials, algebra of polynomial sequences, three step recurrence

\section{Introduction and preliminaries}
\noindent Let us consider a Markov process $(X_t)_{t\geq0}$ with transition probabilities $\mathds{P}_{s,t}(x,\mathrm{d}y)$ for $(s,t)\in\Gamma:=\{(r,u):0\leq r\leq u\}$ and $x\in\supp(X_s)\subseteq\mathds{R}$. Following \cite{hurtado} and \cite{HS}, we say that $(X_t)_{t\geq0}$ is an $m$-polynomial process, $m\in\mathds{N} \cup\{0\}$, if for all $k\in\{0,1,\ldots,m\}$ and any polynomial $f$ of degree at most $k$ the following two conditions hold:
\begin{equation}\label{wart_oczekiwana}
        \mathds{E}(f(X_t)|X_s=x)=\int_\mathds{R}f(y)\mathds{P}_{s,t}(x,\mathrm{d}y)
    \end{equation}
   is a polynomial in variable $x$ of degree at most $k$, and 
    \begin{equation}\label{funkcja}
        (s,t)\mapsto\mathds{E}(f(X_t)|X_s=x)
    \end{equation}
    is in $\mathcal{C}^1(\Gamma)$, i.e., it is a continuously differentiable function in the interior of $\Gamma$, and there exist continuous extensions of this function and its derivatives to the boundary.

\begin{definition}\label{definicja_proces_wiel}
If $(X_t)_{t\geq0}$ is $m$-polynomial process for all $m\geq0$, then it is called a polynomial process.
\end{definition}
\noindent Polynomial processes were first introduced by Cuchiero \cite{Cuchiero} in a time-homogeneous case (then in the second point of the definition of the $m$-polynomial process it is enough to assume that \eqref{funkcja} is in $\mathcal{C}(\Gamma)$ instead of $\mathcal{C}^1(\Gamma)$). Applications of polynomial processes in financial and insurance mathematics, see e.g.  \cite{cuchiero_keller-ressel_teichmann_2012} triggered intensive studies of these processes in recent ten years. For recent extensions of polynomial processes to more abstract settings, see  \cite{cuchiero_larsson_salvuto-ferro_2019}, \cite{cuchiero-svaluto}, \cite{benth-detering-kruhner} and \cite{benth-detering-kruhner0}. 

In the above-mentioned references, much effort has been devoted to studying the properties of infinitesimal generators of polynomials processes, in particular, to relate them to the martingale problem in order to simplify the calculation of some expectations. In this context, giving explicit formulas for infinitesimal generators for the broadest possible class of polynomial processes is of considerable interest.

In this paper, we study a wide sub-class of polynomial processes called quadratic harnesses, which were introduced in \cite{bib_BrycMatysiakWesolowski}. Quadratic harnesses are square-integrable Markov processes with conditional expectations and conditional variances with respect to past-future sigma fields in a  linear and quadratic form, respectively.
More precisely, we assume that $(X_t)_{t\geq0}$ is a separable square-integrable stochastic process such that 
\begin{equation}\label{srednia+cov}
\mathds{E}X_s=0 \quad \textnormal{ and } \quad  \mathds{E}X_tX_s=\min\{s,t\} \quad \textnormal{ for } s,t\geq0.
\end{equation}
 Let $\mathcal{F}_{s,u}:=\sigma\{X_t: t\in[0,s]\cup[u,\infty)\}$, $(s,u)\in\Gamma$, be a natural past-future  filtration generated by the process. We say that $(X_t)_{t\geq0}$ is a quadratic harness if for all $0\leq s<t<u$ we have
\begin{equation}\label{def_harness}
\mathds{E}\big(X_t|\mathcal{F}_{s,u}\big)=a_{tsu}X_s+b_{tsu}X_u
\end{equation}
and
\begin{equation}\label{eq:kwadrat}\mathds{E}\big(X_t^2|\mathcal{F}_{s,u}\big)=A_{tsu}X_s^2+B_{tsu}X_sX_u+C_{tsu}X_u^2+D_{tsu}X_s+E_{tsu}X_u+F_{tsu},\end{equation}
where $a_{tsu}$, $b_{tsu}$, $A_{tsu}, \ldots F_{tsu}$ are deterministic functions depending only on $0\leq s<t<u$. In view of \eqref{srednia+cov} it is easy to see that $a_{tsu}=\tfrac{u-t}{u-s}$ and $b_{tsu}=\tfrac{t-s}{u-s}$.  It is well-known, see \cite[Theorem 2.2]{bib_BrycMatysiakWesolowski}, that under mild technical assumptions $A_{tsu},\ldots,F_{tsu}$ are explicitly identified in terms of five numerical constants 
\begin{equation}\label{par_QH}
\eta, \theta\in\mathds{R}\quad  \sigma\geq0,\quad \tau\geq0\quad \text{ and }\quad q\leq 1+2\sqrt{\sigma\tau}  
\end{equation}
in such a way that
\begin{equation}\label{def_QH}
\begin{split}
\var(X_t|\mathcal{F}_{s,u})=\tfrac{(u-t)(t-s)}{\sigma su+u-qs+\tau}K\left(\tfrac{X_u-X_s}{u-s},\tfrac{uX_s-sX_u}{u-s}\right),
\end{split}
\end{equation}
where $K(x,y)=1+\eta y+\theta x+\sigma y^2+\tau x^2-(1-q)xy$. Typically, the distribution of a quadratic harness process is uniquely determined by conditions \eqref{srednia+cov}, \eqref{def_harness} and \eqref{eq:kwadrat}, and thus by the five constants in \eqref{par_QH}, compare with a discussion after Theorem 2.4 in \cite{bib_BrycMatysiakWesolowski} and the construction of quadratic harnesses \cite{bib_BrycWesolowski_2010} to see that the supports are indeed compact in most cases. Therefore, we will write $(X_t)_{t\geq0}\sim QH(\eta,\theta;\sigma,\tau;q)$ referring to the quadratic harness with the corresponding parameters. Well-known examples of quadratic harnesses are Wiener, Poisson, or Gamma processes. This class also includes classical versions of the free Brownian motion (see \cite{Biane_1998}), the $q$-Gaussian process (see \cite{Bozejko_Kummerer_Speicher_1997}), and the $q$-L\'evy-Meixner process (see \cite{Anshelevich_2004}), the radial part of the quantum Bessel process (see \cite{Biane_1998_Q}). It also contains a wide family of Askey-Wilson processes introduced in \cite{bib_BrycWesolowski_2010}. The latter is of special interest due to its relation to the ASEP (asymmetric simple exclusion process) with open boundaries through the representation of the generating function of the stationary law of the ASEP through joint moments of the Askey-Wilson (or quadratic harness) processes derived in \cite{bib_BrycWesolowski_2017}. The representation was one of the major tools recently used in \cite{CK} (see also \cite{C22}) to identify the multipoint Laplace transform of the stationary measure for the open KPZ equation, as well as to identify Brownian excursions and Brownian meanders as limiting processes in ASEPs of increasing sizes in \cite{BrycWang_2017}. Moreover, in \cite{bib_BrycWesolowski_2017}, the representation combined with the form of the infinitesimal generator of the respective quadratic harness allowed us to derive the formula for the so-called profile density for the ASEP. 

In this paper we will consider quadratic harnesses with $\sigma=0$ (the case when $\sigma\neq 0$ looks technically much more involved and needs a separate investigation). Then $\mathds{E}|X_t|^r<\infty$ for all $r,t>0$ by \cite[Theorem 2.5]{bib_BrycMatysiakWesolowski} and therefore $(X_t)_{t\geq0}$ is an uniquely determined Markov process satisfying conditions \eqref{srednia+cov}, \eqref{def_harness} and \eqref{def_QH}.
Since $\lim\limits_{u\to\infty}\tfrac{X_u}{u}=0$ a.s., see \cite[(2.9)]{bib_BrycMatysiakWesolowski}, we have from \eqref{def_harness} and \eqref{def_QH} 
\begin{equation}\label{martyngal}
    \mathds{E}(X_t|\mathcal{F}_s)=X_s \quad \text{(martingale property)}
\end{equation}
and
\begin{equation}\label{drugi_moment}
    \mathds{E}(X_t^2|\mathcal{F}_s)=X_s^2+\eta(t-s)X_s+(t-s),
\end{equation}
where $(\mathcal{F}_s)_{s\geq0}$ is a natural filtration associated to this process. Moreover, one can calculate all conditional moments $\mathds{E}(X_t^n|\mathcal{F}_s)$ and show that they are polynomials in variable $X_s$ of degree at most $n$, $n\in\mathds{N}$. As a result, $(X_t)_{t\geq0}$ is a time-inhomogeneous polynomial process according to  Definition \ref{definicja_proces_wiel}. 

There is an analogous situation regarding conditioning with respect to the future of the process, i.e., $\mathds{E}(X_t^n|\mathcal{F}_{0,u})$ are also polynomials in variable $X_u$ of degree at most $n$, $n\in\mathds{N}$, so quadratic harnesses are polynomial processes not only with respect to the past but also with respect to the future of the process.

Our main goal in this paper is to derive explicit formulas for infinitesimal generators of $QH(\eta,\theta;0,\tau;q)$. It turns out that infinitesimal generators of quadratic harnesses acting on polynomials (or on a bounded continuous function with bounded continuous second derivative)  can be represented as integro-differential operators, where the integrals are taken with respect to the orthogonality measure of a concrete system of orthogonal polynomials, which is identified through its Jacobi matrix.

Since quadratic harnesses are non-homogeneous Markov processes,   infinitesimal generators are indexed by the time variable $t\ge 0$. To recall the general definition, denote by $\{\mathds{P}_{s,t}(x,\mathrm{d}y): x\in\mathds{R}, 0\leq s <t\}$ the transition probabilities of a non-homogeneous Markov process. The weak left infinitesimal generator is defined by
\begin{equation*}
	\mathbf{A}_t^-f(x):=\lim\limits_{h\to 0^+}\int_\mathds{R}\tfrac{f(y)-f(x)}{h}\mathds{P}_{t-h,t}(x,\mathrm{d}y)
\end{equation*}
for $t>0$ and $f$ such that this pointwise limit exists. 
Analogously, the weak right infinitesimal generator is given by
\begin{equation}\label{gen_1_wstep}
    \mathbf{A}_t^+f(x):=\lim\limits_{h\to 0^+}\int_\mathds{R}\tfrac{f(y)-f(x)}{h}\mathds{P}_{t,t+h}(x,\mathrm{d}y).
\end{equation}
Typically,  for a quadratic harness $(X_t)_{t\ge 0}$ there  exists a family  of martingale orthogonal polynomials $(p_n(\cdot,t))_{n\ge 0}$, $t\ge 0$, i.e., for any $n\ge 0$
$$
\mathds E(p_n(X_t,t)|\mathcal F_s)=p_n(X_s,s),\quad 0\le s<t
$$
and for any $t\ge 0$
$$
\mathds E\,p_n(X_t,t)p_m(X_t,t)=\kappa_n\,\delta_{n=m},\quad n,m\ge 0.
$$
An explicit three-term recurrence for polynomials $(p_n(\cdot,t))_{n\ge 0}$ is given in \cite[Theorem 4.5]{bib_BrycMatysiakWesolowski}.  
Moreover, $\mathbf A_t^\pm(p_n(x,t))=-\tfrac{\partial p_n(x,t)}{\partial t}$, which in particular means that on the domain of polynomials $\mathbf A_t^+$ and $\mathbf A_t^-$ coincide - for details see \cite[Section 1.4]{bib_BrycWesolowski_2015}. Therefore we will use the same symbol $\mathbf A_t$ for both of them in the sequel. Furthermore, considering a Banach space of polynomials up to degree $m\in\mathds{N}\cup\{0\}$ with the proper norm,  \cite[Proposition 2.2.10]{hurtado} says that  point-wise convergence \eqref{gen_1_wstep} for polynomial $f$ implies also convergence in norm.

Hurtado \cite[Lemma 2.2.8]{hurtado} proved that for polynomial processes the infinitesimal generator \eqref{gen_1_wstep} has the  form
$$\mathbf{A}_t^+f(x)=\sum\limits_{l=0}^k\tfrac{\partial}{\partial t}\alpha_l^f(s,t)x^l|_{s=t},$$
where $x\in\text{supp}(X_t)$, $f$ is a polynomial of degree at most $k\geq0$ and $\alpha_0^f,\ldots, \alpha_k^f$ are coefficients occurring on the right hand side of \eqref{wart_oczekiwana}, i.e., 
$$\mathds{E}(f(X_t)|X_s=x)=\sum\limits_{l=0}^k\alpha_l^f(s,t)x^l.$$

We seek rather for a more explicit formula for the infinitesimal generator of quadratic harnesses. Over the years, there have been several different approaches to deriving explicit formulas for infinitesimal generators of the quadratic harness with different restrictions on the parameters $\eta,\theta,\sigma,\tau,q$, see \cite{biane},  \cite{bib_Anshelevich}, \cite{bib_Bryc}, \cite{bib_BrycWesolowski_2014} (generalized later in \cite{bib_BrycWesolowski_2017}) and \cite{bib_BrycWesolowski_2015}. They all lead to the representation of the infinitesimal generator $\mathbf A_t$ as an integro-differential operator of the form
\begin{equation}\label{ATF}
\mathbf A_t\,f\,(x)=\int_\R\,\tfrac{\partial}{\partial x}\left(\tfrac{f(y)-f(x)}{y-x}\right)\,\nu_{x,t}(dy),
\end{equation}
where $\nu_{x,t}$ is some measure. 
Consequently, to determine  $\mathbf A_t$ one has to identify the measure  $\nu_{x,t}$. In particular, it is easy to check that for the Wiener process the above representation holds with $\nu_{x,t}=\delta_x$.

The methodology we propose in this paper has been inspired by quite distinct approaches from   \cite{bib_BrycWesolowski_2014} and \cite{bib_BrycWesolowski_2015}. We now briefly discuss them.

In \cite{bib_BrycWesolowski_2014}, the authors introduced a system of so-called associated (orthogonal) polynomials to the system of polynomials orthogonal with respect to a measure that is a limiting version of transition probabilities of the quadratic harness. Knowing how the infinitesimal generator acts on martingale polynomials allowed deducing a formula for the infinitesimal generator in terms of the orthogonality measure of the system of the associated polynomials. In this way the measure $\nu_{x,t}$ in \eqref{ATF} for the $q$-Meixner-L\'evy processes, $QH(0,\theta;0,\tau;q)$, was identified. In  \cite{bib_BrycWesolowski_2017}, a similar approach was used for a derivation of $\nu_{x,t}$ in \eqref{ATF} for the bi-Poisson process $QH(\eta,\theta;0,0;q)$. In both cases, $\nu_{x,t}$ was expressed in terms of special transition probabilities of the process.

The approach, proposed in \cite{bib_BrycWesolowski_2015}, refers directly to the fact that quadratic harnesses are polynomial processes. It turns out that if $QH(\eta,\theta;\sigma,\tau;q)$ has all moments, then the infinitesimal generator $\mathbf{A}_t$  can be expressed in terms of a special element of certain non-commutative algebra $\mathcal Q$ of infinite polynomial sequences. To identify this element, i.e., to identify the generator, one needs to solve a $q$-commutation equation in $\mathcal Q$. In  \cite{bib_BrycWesolowski_2015} this $q$-commutation equation was solved in case $q=-\sqrt{\sigma\tau}$. Consequently, the infinitesimal generator for the free quadratic harness $X\sim QH(\eta,\theta;\sigma,\tau;-\sqrt{\sigma\tau})$ was identified in the form \eqref{ATF} with the explicit measure $\nu_{x,t}$,  related to the transition probabilities of the process $X$.

In particular, in this case  the authors were able to postulate a special parametric form of the generator. We expect that such approach is not possible in general. Instead, we develop a more universal algebraic approach incorporating associated polynomials which were used more systematically in  \cite{bib_BrycWesolowski_2014}. As a consequence, the approach we propose not only covers all already known cases when $\sigma=0$, but also allows us to derive the infinitesimal generators in new cases of $\tau>0$ and $\eta\neq 0$. We expect that this method extends to the much harder case of $\sigma>0$, but it is beyond the scope of the present investigations.

Here is our main result.
\begin{theorem}\label{twierdzenie_o_generatorze_wiel}
	Let $\mathbf{A}_t$ be the infinitesimal generator of $QH(\eta,\theta;0,\tau;q)$ at time $t\geq0$. Then for every polynomial $f$ and $x\in\text{supp}(X_t)$ 
	\begin{equation}\label{generator}
	    \mathbf{A}_t f(x)=(1+\eta x)\int\limits_{\mathds{R}}\frac{\partial}{\partial x} \frac{f(y)-f(x)}{y-x}\;\nu_{x,t,\eta,\theta,\tau,q}(\mathrm{d}y),
	\end{equation}
	where $\nu_{x,t,\eta,\theta,\tau,q}$ is the probabilistic  orthogonality measure for $(B_n(\cdot;x,t))_{n\geq0}$ defined through a three-step recurrence:
\begin{equation}\label{rekurencja_na_wiel_ort_B}
\begin{split}
B_{-1}(y;x,t)=&0,\quad\quad B_0(y;x,t)=1,\\
y B_n(y;x,t)=&B_{n+1}(y;x,t)\\
&+\big((\gamma_t+\beta_t([n+1]_q+[n]_q))[n+1]_q+x q^{n+1}\big)B_n(y;x,t) \\
&+\alpha_t\big(1+\eta\gamma_t[n]_q+\eta\beta_t[n]_q^2+x\eta q^{n}\big)[n+1]_q[n]_q B_{n-1}(y;x,t), \quad n\geq 0,
\end{split}
\end{equation}
with 
\begin{equation}\label{oznaczenia}
\alpha_t:=\tau+(1-q)t, \quad \beta_t:=\eta \alpha_t, \quad \gamma_t:=\theta-\eta t.
\end{equation}
\end{theorem}
\noindent Above we have used a $q$-notation:
\begin{equation}\label{qnotacja}
   [n]_q=1+q+\ldots+q^{n-1} \text{ for } n\geq1 \text{ and, by  convention, } [0]_q=0.
\end{equation}

The rest of the paper is organized as follows. In the next section we give an overview of an algebra $\mathcal Q$ of polynomial sequences, introduce some special elements of $\mathcal Q$ and analyze their properties. In Section \ref{Sekcja_o_generatorze} we attempt to solve a $q$-commutation  equation in $\mathcal Q$, which is crucial for the identification of the measure $\nu_{x,t}$ in  \eqref{ATF}. To do this, we introduce and carefully examine a  $\mathcal Q$-valued  function $\B$ of real arguments. In Section \ref{dowod_generator} we give a proof of Theorem \ref{twierdzenie_o_generatorze_wiel} and discuss its conclusions.  In Section \ref{specjalne wlasnosci}, we analyze properties of $\nu_{x,t}$ and, in special cases, we relate this measure to transition probabilities of the bi-Poisson process introduced in \cite{bib_BrycMatysiakWesolowski_1}.

\section{Algebra $\mathcal{Q}$ of polynomial sequences}
The algebra of polynomial sequences $\mathcal Q$ was introduced in  \cite{bib_BrycWesolowski_2015} in order to study the properties of polynomial processes. It is defined as a linear space of all infinite sequences of polynomials in a real variable $x$ with a non-commutative multiplication $\mathbb{R}=\mathbb{P}\mathbb{Q}$ for $\mathbb{P}=(P_0,P_1,\ldots)$, $\mathbb{Q}=(Q_0,Q_1,\ldots)$ and $\mathbb{R}=(R_0,R_1,\ldots) \in \mathcal{Q}$ given by
\begin{equation}\label{def_mnozenia}
R_k(x)=\sum\limits_{j=0}^{\textnormal{deg}(Q_k)}[Q_k]_jP_j(x),\quad k\geq 0,
\end{equation}
where $[Q_k]_j$ is the coefficient of $x^j$ in the polynomial $Q_k$. Note that $\mathcal{Q}$ is an algebra with identity $$\E=(1,x,x^2,x^3,\ldots).$$ 
For $a\in\mathds R$ we denote $\E_a=(1,a,a^2,\ldots)$, i.e. all coordinate polynomials (as functions of the generic variable $x$) are of degree zero. Note that for $\P\in\mathcal Q$ we have
\begin{equation}\label{Ea}
\E_a\P=\P|_{x:=a}=(P_0(a),P_1(a),\ldots).
\end{equation}
If $\mathrm{deg}\,P_n=n$ for all $n\ge 0$, then $\mathbb{P}=(P_0,P_1,\ldots)$ is invertible in $\mathcal Q$, see \cite[Proposition 1.2.]{bib_BrycWesolowski_2015}. In the remark below, we give an explicit formula for the inverse of $\E+\X$, where $\X\in\mathcal Q$ satisfies some special conditions.
\begin{remark}\label{uwaga_o_odwrotnosciach_wstep}
    Let $\X\in\mathcal{Q}$ be such that its $n$th coordinate polynomial is of degree at most $n-1$ for $n\geq0$ (where $0$ is a polynomial of degree $-1$). Then $\E+\X\in\mathcal Q$ is invertible and \begin{equation}\label{odwrotka}
    (\E+\X)^{-1}=\sum\limits_{k=0}^\infty(-\X)^k.
    \end{equation}
\end{remark}
\begin{proof}
	Note that the infinite sum on the right-hand side of \eqref{odwrotka} is a well-defined element of $\mathcal Q$ due to the assumption about the degrees of the coordinate polynomials of $\X$. Indeed, \eqref{def_mnozenia} implies that the $n$th coordinate polynomial of $\X^k$ has a degree at most $n-k$ if $0\leq k<n$ or $-1$ otherwise. Hence the sum is finite coordinate-wise. Furthermore,
	$$(\E+\X)\sum\limits_{k=0}^\infty(-\X)^k=\sum\limits_{k=0}^\infty(-\X)^k-\sum\limits_{k=0}^\infty(-\X)^{k+1}=\E,$$
	$$\sum\limits_{k=0}^\infty(-\X)^k(\E+\X)=\sum\limits_{k=0}^\infty(-\X)^k-\sum\limits_{k=0}^\infty(-\X)^{k+1}=\E.$$
\end{proof}

We single out two elements of $\mathcal Q$:
$$\D:=(0,1,x,x^2,\ldots)\qquad \mbox{and}\qquad \F:=(x,x^2,x^3,\ldots),$$
which will play a basic role in the sequel. It is easy to verify that 
\begin{equation}\label{DF=E}
\D\F=\E,
\end{equation}
but $\E-\F\D=(1,0,0,\ldots)$, so $\D$ and $\F$ do not commute.

 Furthermore, for $q\in[-1,1]$ denote
\begin{equation}\label{def_Dq}
\Dq:=\sum\limits_{k=0}^\infty q^k\F^k\D^{k+1},
\end{equation}
an element of $\mathcal Q$, which is important in the analysis below. Clearly, $\D_0=\D$. Note that $\Dq$ can be written in terms of q-notation as $$\Dq=([0]_q,[1]_q,[2]_q x,[3]_q x^2,\ldots),$$
recall \eqref{qnotacja}. When applied from the left to $\mathbb P\in\mathcal Q$  it acts coordinate-wise as the $q$-derivative. In particular, for $q=1$ we have
\begin{equation}\label{D1}
\D_1:=\sum\limits_{k=0}^\infty \F^k\D^{k+1},
\end{equation}
and $\D_1$ represents the classical derivative. Moreover, $\Dq$ satisfies the following identities:
\begin{equation}\label{wlasnosci_Dq_1}
   \Dq(\E-\F\D)=\0,
\end{equation}
\begin{equation}\label{wlasnosci_Dq_2}
    \Dq\F=q\F\Dq+\E.
\end{equation}
The former, \eqref{wlasnosci_Dq_1}, is satisfied due to \eqref{def_Dq} and \eqref{DF=E}. The latter, \eqref{wlasnosci_Dq_2}, follows from
$$\Dq\F-q\F\Dq=\sum\limits_{k=0}^\infty q^k\F^k\D^k-\sum\limits_{k=0}^\infty q^{k+1}\F^{k+1}\D^{k+1}=\E.$$

For future use, it will be convenient to introduce some identities involving additional special elements of $\mathcal Q$ that will be important in the derivation of the Jacobi matrix of the orthogonal polynomials related to the measure building up the infinitesimal generators we search for.
\begin{lemma}
	For $\beta\in \mathds{R}$ and
	\begin{equation*}
      \begin{split}
        \W_1 &:= \E+\beta\Dq\F\Dq,\\
        \W_3 &:= \E+q\beta\F\Dq^2,
      \end{split}
    \qquad\qquad
      \begin{split}
        \W_2 &:= \E+\beta\F\Dq^2,\\
        \W_4&:= \E+\beta\Dq^2\F
      \end{split}
    \end{equation*}
the following identities hold 
	\begin{multicols}{2}
		\begin{enumerate}[label=(\roman*)]
		    \item $\W_1\Dq=\Dq\W_2$,
			\item $\Dq\W_1=\W_4\Dq$,
			\item $\Dq^2\W_2=\W_4\Dq^2$,
			\item $\W_1=\W_3+\beta\Dq$,
			\item $\Dq\F\W_3=\E+q\W_2\F\Dq$,
			\item $\W_2\W_3=\W_1(\W_2-\beta\Dq)$.
		\end{enumerate}
	\end{multicols}
	\label{lemat_wlasnosci_elementow_algebry}
\end{lemma}
\begin{proof}
	Identities $(i)$ and $(ii)$ follow directly from the definitions of $\W_1$, $\W_2$ and $\W_4$, while $(iii)$ is a combination of $(i)$ and $(ii)$. Identity $(iv)$ is an immediate consequence of  \eqref{wlasnosci_Dq_2}. To see $(v)$, note that from \eqref{wlasnosci_Dq_2} we have
	$\Dq\F^2\Dq=\F\Dq(q\F\Dq)+\F\Dq=\F\Dq^2\F$ and $$\Dq\F\W_3=\Dq\F+q\beta\Dq\F^2\Dq^2=\E+q\F\Dq+q\beta\F\Dq^2\F\Dq=\E+q\W_2\F\Dq.$$
	Finally we use $(iv)$ and $(i)$ to show $(vi)$ as follows
	\begin{equation*}
	\W_2\W_3=\W_3\W_2=\big(\W_1-\beta\Dq\big)\W_2=\W_1\W_2-\beta\Dq\W_2=\W_1(\W_2-\beta\Dq).
	\end{equation*}
\end{proof}

\begin{remark}\label{uwaga_o_odwrotnosciach}
    Elements $\W_i$, $i=1,2,3,4$, are invertible and their inverses can be expressed by \eqref{odwrotka}, where $\X$ is equal $\beta\Dq\F\Dq$, $\beta\F\Dq^2$, $q\beta\F\Dq^2$, $\beta\Dq^2\F$, respectively.
\end{remark}
The above follows from Remark \ref{uwaga_o_odwrotnosciach_wstep}, since the $n$-th coordinate polynomials of $\beta\Dq\F\Dq$, $\beta\F\Dq^2$, $q\beta\F\Dq^2$, $\beta\Dq^2\F$  is equal to  $\beta[n]_q^2x^{n-1}$, $\beta[n]_q[n-1]_qx^{n-1}$, $q\beta[n]_q[n-1]_qx^{n-1}$, $\beta[n+1]_q[n]_qx^{n-1}$, respectively, for  $n\geq0$ (recall that $[0]_q=0$). 

The elements $\W_1$, $\W_2$, and $\Dq$ are basic building blocks of more complicated elements of $\mathcal{Q}$, which are important for our considerations. 
For  real coefficients $\alpha,\,\beta,\,\gamma$ let us define 
\begin{equation}\label{def S}
\Sz:=\R+z(\D-\Q),\quad z\in\mathds R,
\end{equation}
where 
\begin{align}\label{def R}
    &\R:=(\W_1+\gamma\Dq)\W_2+\alpha\Dq^2,\\
    &\Q:=(1-q)\Dq\W_2-\beta\Dq^2\label{def Q}.
\end{align}
As we will see later, $\Sz$ plays the role of the Jacobi matrix of a system of orthogonal polynomials that will allow us to identify the orthogonality measure that is the basic ingredient of the infinitesimal generator of the process. In the lemma and its corollary we present two identities satisfied by $\Sz$, $\R$, and $\Q$.
\begin{lemma}\label{lemat_o_komutacji_R_i_Q} 
	For $\X\in\{\R,\Q\}$ we have
	\begin{equation}\label{RQRQ}\Dq\W_1^{-1}\X\F\W_2\W_3=\X(\E+q\F\Dq\W_2).
	\end{equation}
\end{lemma}
\begin{proof}
We first note that identity
\begin{equation}\label{dqf}
	\W_4\Y\W_2^{-1}\D_q=\Dq\Y
	\end{equation}
	holds for (1) $\Y:=\W_1\W_2$, (2) $\Y:=\Dq\W_2$ and (3) $\Y:=\D_q^2$.  Case (1) holds due to (i) and (ii), cases (2) and (3) due to (iii), see Lemma \ref{lemat_wlasnosci_elementow_algebry}.
	
	Since $\R$ and $\Q$ are linear in $(\W_1\W_2,\Dq\W_2,\Dq^2)$, we conclude that \eqref{dqf} holds also for $\Y:=\X$. Consequently,
	$$
	\Dq\X\F\W_2\W_3=\W_4\X\W_2^{-1}\D_q\F\W_2\W_3,
	$$
	but  (v) of Lemma \ref{lemat_wlasnosci_elementow_algebry} implies
	$$
	\W_2^{-1}\Dq\F\W_2\W_3=\E+q\F\Dq\W_2.
	$$
	Thus we obtain
	\begin{equation*}
	\Dq\X\F\W_2\W_3=\W_4\X(\E+q\F\Dq\W_2).
	\end{equation*}
	which, due to (ii) of Lemma \ref{lemat_wlasnosci_elementow_algebry} and Remark \ref{uwaga_o_odwrotnosciach}, is equivalent to \eqref{RQRQ}.
\end{proof}
\begin{corollary}\label{wniosek_o_Sz}
   For $z\in\mathds{R}$ we have
    \begin{equation*}
    \Dq\W_1^{-1}\Sz\F\W_2\W_3=\R+q\Sz\F\Dq\W_2.
    \end{equation*}
\end{corollary}
\begin{proof}
    Writing $\Sz=(\R-z\Q)+z\D$ and using Lemma \ref{lemat_o_komutacji_R_i_Q} we see that the left hand side above is
    $$\R+q\Sz\F\Dq\W_2-z(\Q+q\D_q\W_2-\D_q\W^{-1}_1\W_2\W_3).$$
    Due to (vi) of Lemma \ref{lemat_wlasnosci_elementow_algebry}, the last term in the bracket is $\D_q\W_2-\beta\D_q^2$. Thus, recalling the definition of $\Q$, we see that the whole bracket vanishes.
\end{proof}

Recall that $X_n$ denotes the $n$th coordinate polynomial of  $\X\in\mathcal{Q}$, $n\geq0$. If $\X\in\mathcal Q$ additionally depends on a parameter $z\in\mathds R$ we write $\X(z)$ as for $\Sz$ above. Then its $n$th coordinate, denoted by $X_n(z)$, is a polynomial in the generic variable $x$ with coefficients depending on $z$, $n\ge 0$. In the sequel, we need to evaluate the product $\X(z)\Y(z)$ at $z:=x$. It is easy to see that even when $\X(z)|_{z:=x},\,\Y(z)|_{z:=x}\in\mathcal Q$ the identity $(\X(z)\Y(z))|_{z:=x}=\X(z)|_{z:=x}\,\Y(z)|_{z:=x}$ may not hold. 
\begin{remark}\label{uwaga_o_funkcjach}
	Let $\X(z),\Y(z)\in\mathcal Q$ for all $z\in\mathds{R}$. Assume that all coefficients of the coordinate polynomials $X_n(z)$ (in the generic variable $x$) are polynomials in $z$, $n\ge 0$. Then  $\X(z)|_{z:=x}\in\mathcal Q$ and 
	$$(\X(z)\Y(z))|_{z:=x}=(\X(z)|_{z:=x}\Y(z))|_{z:=x}.$$
\end{remark}
\begin{proof} By the assumption on the coefficients of the coordinate polynomials of $\X(z)$, we conclude that $X_n(z)|_{z:=x}$, which is the $n$th coordinate of $\X(z)$ evaluated at $z:=x$, is a polynomial in $x$, i.e., $\X(z)|_{z:=x}\in\mathcal Q$.

Note that $Y_n(z)$, the $n$th coordinate polynomial of $\Y(z)$, can be written as $Y_n(z)=\sum\limits_{k=0}^{M_n}y_{k,n}(z)x^k$ for some $M_n\in\mathds{N}\cup\{0\}$ and $(y_{k,n}(z))_{k=0,\ldots,M_n}\in \mathds{R}^{M_n+1}$, $n\ge 0$. Then, by \eqref{def_mnozenia}, the $n$th coordinate of $\X(z)\Y(z)$ is equal $\sum\limits_{k=0}^{M_n}y_{k,n}(z)X_k(z)$ whereas the $n$th coordinate of $\X(x)\Y(z)$ equals $\sum\limits_{k=0}^{M_n}y_{k,n}(z)X_k(x)$. Therefore, considering these two objects as functions of $z$ and inserting $z:=x$ yields the desired equality coordinate-wise.
\end{proof}

\section{Infinitesimal generator as an element of the algebra}\label{Sekcja_o_generatorze}
As already mentioned, $\mathbf{A}_t$ evaluated on polynomials also gives polynomials. Thus, $\mathbf{A}_t$ has the unique representation $\A_t\in\mathcal{Q}$ with the $n$th coordinate equal to $\mathbf{A}_t(x^n)$ for $n\geq0$, compare with \cite[Section 1.4]{bib_BrycWesolowski_2015}. It was also explained in \cite[Theorem 2.1]{bib_BrycWesolowski_2015} that the infinitesimal generator $\A_t$ of the quadratic harness $X\sim QH(\eta,\theta;\sigma,\tau;q)$ can be identified through a solution $\H_t\in\mathcal{Q}$ of the following $q$-commutation equation for $t\geq0$:
\begin{equation}\label{generator1}
\H_t\T_t-q\T_t\H_t=\E+\theta\H_t+\eta \T_t+\tau\H_t^2+\sigma\T_t^2,
\end{equation}
where $\T_t:=\F-t\H_t$, with the initial condition $\H_t(\E-\F\D)=\0$.It has been proved, see \cite[Proposition 2.4]{bib_BrycWesolowski_2015},  that \eqref{generator} has a unique solution $\H_t$ when $\sigma\tau\ne1$. If the solution $\H_t$ is found, the generator $\A_t$ can be recovered from the formula $\H_t:=\A_t\F-\F\A_t$. This plan was successfully realized in \cite[Theorem 2.5]{bib_BrycWesolowski_2015} for the free quadratic harness, i.e., $QH(\eta,\theta;\tau,\sigma;-\tau\sigma)$ and for the classical version of the quantum Bessel process (see \cite{biane_fran}), i.e. $QH(\eta,\theta;0,0;1)$,  in \cite[Section 5.2]{bib_BrycWesolowski_2015}. In the former case, the solution of \eqref{generator1} is given as  
$$
\H_t=\tfrac{1}{1+\sigma t}(\E+\eta \F+\sigma \F^2)\,\phi_t(\D)\D,
$$
where $\phi_t$ is a function satisfying certain quadratic equation (see \cite[Lemma 3.1]{bib_BrycWesolowski_2015}).  In the latter case, the solution of \eqref{generator1}  was shown to be
$$
\H_t=\tfrac{1}{\theta-t\eta}\,(\E+\eta \F)\left(e^{(\theta-t\eta)\D_1}-\E\right),
$$
where $\D_1$, defined by \eqref{def_Dq} with $q=1$, acts as the classical derivative.

In this paper, we extend this algebraic method to all quadratic harnesses with $\sigma=0$. The case of $\sigma\ne 0$ is more difficult because of the term $\sigma\T_t^2$  in the $q$-commutation equation. Indeed, in the argument presented below, one can see that $\sigma\neq 0$ adds a term $\sigma \F^2$ to the equations we deal with, and it triggers additional difficulties that we do not know how to overcome. Therefore we focus on $\sigma=0$, which is already quite complicated. In this case, the construction of such processes was done under some constraints on parameters, see Section \ref{rozdzial_QH_z_tau} in the Appendix. 

Recall that our primary goal is to solve equation \eqref{generator1}. Consider $\widetilde{\H}_t$ satisfying a similar equation 
\begin{equation}\label{tildeH}
\widetilde{\H}_t\F-q\F\widetilde{\H}_t=\E+\gamma_t\widetilde{\H}_t+\alpha_t\widetilde{\H}_t(\E+\eta\F)\widetilde{\H}_t
\end{equation}
with the initial condition $\widetilde{\H}_t(\E-\F\D)=\0$, where $\alpha_t$ and $\gamma_t$ are given in \eqref{oznaczenia}. For such $\widetilde{\H}_t$ let $\H_t=(\E+\eta\F)\widetilde{\H}_t$.  Obviously, $\H_t(\E-\F\D)=\0$. Upon multiplication \eqref{tildeH} from the left by $\E+\eta\F$, we conclude that $\H_t$ satisfies \eqref{generator1}. Thus, to put the above methodology to work, it suffices to find the solution $\widetilde{\H}_t$ of \eqref{tildeH}. 

Fix $t>0$ and let $R(\X):=\E+\gamma_t\X+\alpha_t \X^2+\beta_t\X\F\X$. We rewrite \eqref{tildeH} in the equivalent form
\begin{equation}
\widetilde{\H}_t\F=q\F\widetilde{\H}_t+R(\widetilde{\H}_t).
\label{HfalkaF}
\end{equation}
Let $\Sz$ be as defined in \eqref{def S}-\eqref{def Q} with $(\alpha,\beta,\gamma)=(\alpha_t,\beta_t,\gamma_t)$, see \eqref{oznaczenia}. From this moment on, since $t$ is fixed, we suppress $t$ in subscripts. Consider a function $\B:\mathds R\to\mathcal Q$ satisfying 
\begin{equation}\label{FB}
	\F\Bz=\Bz\Sz\F,
\end{equation}
with initial condition 
\begin{equation}
\label{initial_cond}
\Bz(\E-\F\D)=\E-\F\D.
\end{equation}
Clearly, \eqref{FB} and \eqref{initial_cond} uniquely determine $\Bz$ in terms of $\alpha,\beta,\gamma$.  Indeed, looking coordinate-wise at \eqref{FB}  we get (after simplification) the three-term recurrence \eqref{rekurencja_na_wiel_ort_B} for polynomials $B_n(\cdot;z,t):=B_n(z)$, $n\geq0$, in the generic variable $x\in\mathds{R}$, and these polynomials are the entries of $\Bz$.  Then we obtain the following result inductively:
\begin{corollary}\label{wniosek_o_funkcji_B}
	For any $z\in \mathds R$ the equations \eqref{FB} and \eqref{initial_cond} correctly define $\Bz\in\mathcal Q$.  Its $n$th coordinate, $B_n(z)$, is a monic polynomial of degree $n$ of the generic variable $x\in \mathds R$, $n\geq0$. Consequently, $\Bz$ is invertible. Furthermore, $B_n(z)$ is also a polynomial of variable $z\in\mathds{R}$, $n\ge 0$.
\end{corollary} 
\noindent Moreover,  \eqref{FB} and \eqref{initial_cond} are equivalent to
\begin{equation}\label{definicja_B}
\F\Bz\D+\E-\F\D=\Bz \Sz.
\end{equation}
Indeed, due to \eqref{wlasnosci_Dq_1} and the definition of $\Sz$, we have 
\begin{equation}\label{Sz(E-FD)}
\Sz(\E-\F\D)=\E-\F\D.
\end{equation}
If we multiply \eqref{FB} from the right by $\D$, we obtain \eqref{Sz(E-FD)} as follows
\begin{equation*}
    \begin{split}
        \F\Bz\D&=\Bz\Sz\F\D=\Bz\Sz-\Bz\Sz(\E-\F\D)\\
        &=\Bz\Sz-\Bz(\E-\F\D)=\Bz\Sz-(\E-\F\D).
    \end{split}
\end{equation*}
Conversely, in view of \eqref{DF=E} equality \eqref{definicja_B} multiplied from the right by $\F$ gives directly \eqref{FB}. Similarly, in view of \eqref{Sz(E-FD)}, multiplication of \eqref{definicja_B} from the right by $\E-\F\D$ gives \eqref{initial_cond}.

Now we are going to derive important relations between  $\Bz$ and $\widetilde{\H}$.
\begin{theorem}\label{glowne_twierdzenie}
	Assume that $\widetilde{\H}$ satisfies \eqref{HfalkaF} and $\widetilde{\H}(\E-\F\D)=\0$. Then for all $z\in\mathds{R}$ we have
	\begin{equation}\label{HB}
	\widetilde{\H}=\Bz\Dq\W_1^{-1}\Bz^{-1}.
	\end{equation}
	Moreover, for  $\widetilde{\M}:=\widetilde{\H}\F-\F\widetilde{\H}$ we get
	\begin{equation}\label{MB}
	\widetilde{\M}\Bz\W_2\W_3=\E-\F\D+(\F-z\E)\Bz(\D-\Q).
	\end{equation}
\end{theorem}
\begin{proof}  In the proof we fix arbitrary $z\in\mathds{R}$ and  we suppress $(z)$ in $\Bz$, $\X(z)$, and $\Sz$ to make the expression easier to follow.
  \begin{enumerate} 
  \item Proof of \eqref{HB}. 
  Since $\W_1$ (see Remark \ref{uwaga_o_odwrotnosciach}) and $\B$ (see Corollary \ref{wniosek_o_funkcji_B}) are invertible,
    \begin{equation}\label{X}
    \X:=\B\Dq\W_1^{-1}\B^{-1}
    \end{equation}
    is a well-defined element of $\mathcal{Q}$. Due to \eqref{initial_cond} we have $\B^{-1}(\E-\F\D)=\E-\F\D$ and according to identity \eqref{wlasnosci_Dq_1} we get $\W_1(\E-\F\D)=\E-\F\D$, thus we finally obtain that
    $$\X(\E-\F\D)=\B\Dq\W_1^{-1}\B^{-1}(\E-\F\D)=\B\Dq\W_1^{-1}(\E-\F\D)=\B\Dq(\E-\F\D)=\0.$$
   Since equation \eqref{HfalkaF}  with the initial condition uniquely determines $\widetilde\H$ it suffices to show that $\X$ defined by \eqref{X}  satisfies 
    \begin{equation}\label{RHS}
    LHS:=\X\F-q\F\X-\E-\gamma\X-\alpha\X^2-\beta\X\F\X=\0,
    \end{equation}
    because we have already checked that $\X(\E-\F\D)=\0$.
    Since \eqref{FB} implies $\B^{-1}\F=\Sb\F\B^{-1}$, after simplification $LHS$ above assumes the form
    \begin{align*}
    LHS=&\B\Dq\W_1^{-1}\Sb\F\B^{-1}-q\B\Sb\F\Dq\W_1^{-1}\B^{-1}-\E-\gamma\B\Dq\W_1^{-1}\B^{-1}\\
    &-\alpha\B(\Dq\W_1^{-1})^2\B^{-1}-\beta\B\Dq\W_1^{-1}\Sb\F\Dq\W_1^{-1}\B^{-1}\\
    =&\B\W\B^{-1},
    \end{align*}
    where 
    $$\W:=\Dq\W_1^{-1}\Sb\F-q\Sb\F\Dq\W_1^{-1}-\E-\gamma\Dq\W_1^{-1}-\alpha(\Dq\W_1^{-1})^2-\beta\Dq\W_1^{-1}\Sb\F\Dq\W_1^{-1}.$$
    Consequently, applying (i) of Lemma \ref{lemat_wlasnosci_elementow_algebry}, we obtain
   \begin{equation*}
   \begin{split}
      \W\W_1\W_2&=\Dq\W_1^{-1}\Sb\F\W_1\W_2-q\Sb\F\Dq\W_2-\W_1\W_2-\gamma\Dq\W_2-\alpha\Dq^2-\beta\Dq\W_1^{-1}\Sb\F\Dq\W_2\\
      &=\Dq\W_1^{-1}\Sb\F(\W_1-\beta\Dq)\W_2-q\Sb\F\Dq\W_2-\R,
    \end{split}
   \end{equation*}
    where the last equation holds due to \eqref{def R}. From (iv) of Lemma \ref{lemat_wlasnosci_elementow_algebry} and Corollary \ref{wniosek_o_Sz} we have
    \begin{equation*}
   \begin{split}
      \W\W_1\W_2&=\Dq\W_1^{-1}\Sb\F\W_3\W_2-q\Sb\F\Dq\W_2-\R=\0.
    \end{split}
   \end{equation*}
   Because $\W_1$ and $\W_2$ are invertible, then $\W=\0$ and consequently \eqref{RHS} holds true.
   \item Proof of \eqref{MB}. 
   Now we consider $\widetilde{\M}$. Using \eqref{FB} and \eqref{HB} we get
   $$\widetilde{\M}=\widetilde{\H}\F-\F\widetilde{\H}=\B\Dq\W_1^{-1}\B^{-1}\F-\F\B\Dq\W_1^{-1}\B^{-1}=\B(\Dq\W_1^{-1}\Sb\F-\Sb\F\Dq\E_{1}^{-1})\B^{-1}.$$
   Therefore by (vi) of Lemma \ref{lemat_wlasnosci_elementow_algebry} and Corollary \ref{wniosek_o_Sz} we get
    \begin{equation*}
    \begin{split}
        \widetilde{\M}\B\W_2\W_3&=\B(\Dq\W_1^{-1}\Sb\F\W_2\W_3-\Sb\F\Dq\W_1^{-1}\W_2\W_3)\\
        &=\B(\R+q\Sb\F\Dq\W_2-\Sb\F\Dq(\W_2-\beta\Dq)).
    \end{split}
    \end{equation*}
    Using the fact that $\R=\Sb-z(\D-\Q)$ and recalling \eqref{def Q}, we can rewrite the above as
    $$\widetilde{\M}\B\W_2\W_3=\B(\R-\Sb\F((1-q)\Dq\W_2-\beta\Dq^2))=\B\Sb-z\B(\D-\Q)-\B\Sb\F\Q.$$ 
    According to \eqref{definicja_B} and \eqref{FB}  we finally obtain 
    $$\widetilde{\M}\B\W_2\W_3=\F\B\D+\E-\F\D-z\B(\D-\Q)-\F\B\Q=\E-\F\D+(\F-z\E)\B(\D-\Q).$$ 
    \end{enumerate}
    \end{proof}
    The following result, a consequence of Theorem \ref{glowne_twierdzenie}, will be crucial in deriving the explicit form of $\mathbf A_t$ in the next section.
    \begin{corollary}\label{Bzx}
    For $\widetilde{\M}$ and $\Bz$ defined above
    \begin{equation}\label{MB_w_x}
    \widetilde{\M}=\left((\E-\F\D)\Bz^{-1}\right)|_{z:=x}.
    \end{equation}
    \end{corollary}
    \begin{proof}
    Directly from \eqref{DF=E} we deduce that $$(\E-\F\D)\W_2\W_3=\E-\F\D.$$
    Then \eqref{MB} multiplied from the right by $\W_3^{-1}\W_2^{-1}\Bz^{-1}$ ($\Bz$ is invertible from Corollary \ref{wniosek_o_funkcji_B}) yields
    $$\widetilde{\M}=\left(\E-\F\D+(\F-z\E)\Bz(\D-\Q)\W_3^{-1}\W_2^{-1}\right)\Bz^{-1}=\X(z)\Y(z)$$
    where $\X(z)=\E-\F\D+(\F-z\E)\Bz(\D-\Q)\W_3^{-1}\W_2^{-1}$ and $\Y(z)=\Bz^{-1}$. Since $\widetilde{\M}$ does not depend on $z$, we can write
    \begin{equation}\label{MXY}
    \widetilde{\M}=(\X(z)\Y(z))|_{z:=x}.
    \end{equation} 
    Note that  
    $$\X(z)|_{z:=x}=\E-\F\D+\left(\widetilde{\X}(z)\widetilde{\Y}(z)\right)|_{z:=x},$$
    where $\widetilde{\X}(z)=\F-z\E$ and $\widetilde{\Y}(z)=\Bz(\D-\Q)\W_3^{-1}\W_2^{-1}$. Since $\widetilde{\X}(z)|_{z:=x}=(\F-z\E)|_{z:=x}=\F-\F=0$, in view of Corollary \ref{wniosek_o_funkcji_B} applied to $\widetilde{\X}(z)$ and $\widetilde{\Y}(z)$, we conclude that $\X(z)|_{z:=x}=\E-\F\D$. Applying Corollary \ref{wniosek_o_funkcji_B} again, this time to $\X(z)$ and $\Y(z)$, the right-hand side of \eqref{MXY} simplifies to \eqref{MB_w_x}. 
    \end{proof}
    Since $\widetilde{\H}\F\D=\widetilde{\H}$, it follows from the definition of $\widetilde{\M}$ that $\widetilde{\H}=\widetilde{\M}\D+\F\widetilde{\H}\D$. Iterating this equality we get
    \begin{equation}\label{Hser}
    \widetilde{\H}=\sum_{k\ge 0}\,\F^k\widetilde{\M}\D^{k+1},
    \end{equation}
    where, coordinate-wise, all sums have a finite number of non-zero summands.\\
    Note that as intended, \eqref{Hser} is the solution $\widetilde{\H}$ of the equation \eqref{tildeH} given in terms of $\Bz$. Consequently, we have also found the solution to the initial $q$-commutation equation \eqref{generator1}. 
\section{Infinitesimal generator of the quadratic harness $QH(\eta,\theta;0,\tau;q)$}\label{dowod_generator}
Equipped with the results obtained in the previous section, we are ready to prove our main result which gives an explicit integral representation of the infinitesimal generator of $QH(\eta,\theta;0,\tau;q)$.

\begin{proof}[Proof of Theorem \ref{twierdzenie_o_generatorze_wiel}]Fix $t\geq0$ and $z\in\mathds{R}$. By Favard's theorem, see \cite[Theorem 4.4.]{Chihara}, polynomials $\{B_n(x;z,t):n\geq0\}$ are "orthogonal" with respect to a unique moment functional $\mathcal{L}_{z,t,\theta,\tau,\eta,q}$, which acts on polynomials in variable $y\in\mathds{R}$, i.e.,
$$\mathcal{L}_{z,t,\theta,\tau,\eta,q}(B_n(y;z,t)B_k(y;z,t))=\kappa_n\mathds{1}(n=k),$$
where $\kappa_0\ne0$. Without any loss of generality, we assume that $\mathcal{L}_{z,t,\theta,\tau,\eta,q}$ is normalized, i.e.,  $\kappa_0=1$. Then for $\E_y\in\mathcal Q$, $y\in\mathds{R}$, see \eqref{Ea}, we obtain
\begin{equation}\label{intB}
\mathcal{L}_{z,t,\theta,\tau,\eta,q}(\E_y\Bz)=\E-\F\D,
\end{equation}
where $\mathcal{L}_{z,t,\theta,\tau,\eta,q}$ on the left hand side of \eqref{intB} acts coordinate-wise on $\E_y\Bz$.
    Moreover, let us consider $\Z(z)$ given by
    \begin{equation*}
    \Z(z):=\bigg(\mathcal{L}_{z,t,\theta,\tau,\eta,q}(1),\mathcal{L}_{z,t,\theta,\tau,\eta,q}(y),\mathcal{L}_{z,t,\theta,\tau,\eta,q}(y^2),\ldots\bigg)=\mathcal{L}_{z,t,\theta,\tau,\eta,q}(\E_y).
    \end{equation*}
    Note that $\Z(z)$ is a well-defined element of the algebra $\mathcal{Q}$ with all coordinates being polynomials of degree zero (in $x$). By the linearity of the moment functional, we obtain, in view of \eqref{intB}, that
    \begin{equation*}
    \Z(z)\Bz=\mathcal{L}_{z,t,\theta,\tau,\eta,q}(\E_y\Bz) =\E-\F\D.
    \end{equation*}
    Hence,
    $$\Z(z)=(\E-\F\D)\Bz^{-1}.$$
    Since the above equality holds for all fixed $z\in\mathds{R}$, a comparison with \eqref{MB_w_x} gives that $\Z(z)|_{z=x}\in\mathcal{Q}$ and
    \begin{equation}\label{M=Z}
    \widetilde\M=\Z(z)|_{z=x}=\mathcal{L}_{z,t,\theta,\tau,\eta,q}(\E_y)|_{z=x}=\mathcal{L}_{x,t,\theta,\tau,\eta,q}(\E_y)
    \end{equation}
     Thus, inserting \eqref{M=Z} into \eqref{Hser}, linearity of $\mathcal{L}_{x,t,\theta,\tau,\eta,q}$ implies
    \begin{equation}\label{wtildeH}
    \widetilde{\H}=\mathcal{L}_{x,t,\theta,\tau,\eta,q}(\Q_y),
    \end{equation}
    where 
    $$\Q_y:=\sum\limits_{k=0}^\infty\,\F^k\E_y\D^{k+1}.$$
	Recall that, see \cite[(3.8)]{bib_BrycWesolowski_2015},
	\begin{equation}\label{At}
	\A=\sum_{j=0}^{\infty}\,\F^j\H\D^{j+1}.
	\end{equation}
    Since $\H=(\E+\eta\F)\widetilde{\H}$, plugging this together with \eqref{wtildeH} into \eqref{At} we obtain
	$$
	\A=\sum_{j=0}^{\infty}\,\F^j(\E+\eta\F)\widetilde{\H}\D^{j+1}=(\E+\eta\F)\mathcal{L}_{x,t,\theta,\tau,\eta,q}\left(\sum_{j=0}^{\infty}\F^j\Q_y\D^{j+1}\right).
	$$
	Since $\D\E_y=\mathbb 0$, then, according to \eqref{D1}, we have 
	\begin{align*}\D_1\Q_y=&\left(\sum_{j=0}^{\infty}\,\F^j\D^{j+1}\right)\left(\sum_{k=0}^{\infty}\,\F^k\E_y\D^{k+1}\right)=\sum_{j=0}^{\infty}\,\F^j\left(\sum_{k\ge j+1}\,\F^{k-j-1}\E_y\D^{k+1}\right)\\
	=&\sum_{j=0}^{\infty}\F^j\left(\sum_{k\ge j+1}\,\F^{k-j-1}\E_y\D^{k-j}\right)\D^{j+1}=\sum_{j=0}^{\infty}\F^j\Q_y\D^{j+1}.
	\end{align*}
	Thus,
	$$\A=(\E+\eta\F)\mathcal{L}_{x,t,\theta,\tau,\eta,q}(\D_1\Q_y).$$
	Since $\E_y\D^{k+1}$ has the $n$th coordinate equal to zero for $n\le k$ and equal to $y^{n-k-1}$ for $n\ge k+1$, the $n$th coordinate of $\Q_y$, has the form  
    \begin{equation*}
    y^{n-1}+y^{n-2}x+\ldots+yx^{n-2}+x^{n-1}=\tfrac{y^n-x^n}{y-x}.
    \end{equation*}
	Therefore, the $n$th coordinate of $\D_1\Q_y$ is $\tfrac{\partial}{\partial x}\,\tfrac{y^n-x^n}{y-x}$. Consequently, $A_n(x)$, the $n$th coordinate of $\A$, assumes the form 
	\begin{equation*}
	A_n(x)=(1+\eta x)\mathcal{L}_{x,t,\theta,\tau,\eta,q}\left(\tfrac{\partial}{\partial x}\,\tfrac{y^n-x^n}{y-x}\right).
	\end{equation*}
	Recall that $A_n(x)=\mathbf A_t(x^n)$. Due to the linearity, we get for any polynomial $f$
	\begin{equation}\label{A_jako_operator}
	\mathbf A_tf(x)=(1+\eta x)\mathcal{L}_{x,t,\theta,\tau,\eta,q}\left(\tfrac{\partial}{\partial x}\,\tfrac{f(y)-f(x)}{y-x}\right).
    \end{equation}
    
Now are going to show that the moment functional $\mathcal{L}_{x,t,\theta,\tau,\eta,q}$  is non-negative. Consider $\mathbb G:=(G_0,G_1,G_2,\ldots)\in\mathcal Q$ with coordinates of the form
$$
G_k(x):=\lim\limits_{h\to0^+}\int\limits_{\mathds{R}}y^k\tfrac{(y-x)^2}{h}\mathds{P}_{t,t+h}(x,\mathrm{d}y),\quad k=0,1,\ldots
$$
For any $k\in\mathds{N}\cup\{0\}$
\small
\begin{equation*}
    \begin{split}
        \int\limits_{\mathds{R}}y^k\tfrac{(y-x)^2}{h}\mathds{P}_{t,t+h}(x,\mathrm{d}y)&=\int\limits_{\mathds{R}}\tfrac{y^{k+2}-x^{k+2}}{h}\mathds{P}_{t,t+h}(x,\mathrm{d}y)-2x\int\limits_{\mathds{R}}\tfrac{y^{k+1}-x^{k+1}}{h}\mathds{P}_{t,t+h}(x,\mathrm{d}y)\\
        &+x^2\int\limits_{\mathds{R}}\tfrac{y^{k}-x^{k}}{h}\mathds{P}_{t,t+h}(x,\mathrm{d}y),
    \end{split}
\end{equation*}
\normalsize
whence 
$$
G_k(x)=\mathbf A_t(x^{k+2})-2x\mathbf A_t(x^{k+1})+x^2\mathbf A_t(x^k),\quad k=0,1,\ldots
$$
In view of the fact that $\H_t=\A_t\F-\F\A_t$, the above equality implies 
$$
\mathbb G=\A_t\F^2-2\F\A_t\F+\F^2\A_t=\H_t\F-\F\H_t.
$$
Consequently, since $\H=(\E+\eta\F)\widetilde \H$ and $\widetilde \H\F-\F\widetilde \H=\widetilde \M$, see \eqref{M=Z}, we obtain
$$
\mathbb G=(\E+\eta\F)(\widetilde \H\F-\F\widetilde \H)=(\E+\eta\F)\mathcal{L}_{x,t,\theta,\tau,\eta,q}(\E_y).
$$
Looking coordinate-wise at the above identity, we get that for all $k\geq0$
\begin{equation}\label{y^k}
\lim\limits_{h\to0^+}\int\limits_{\mathds{R}}y^k\tfrac{(y-x)^2}{h}\mathds{P}_{t,t+h}(x,\mathrm{d}y)=(1+\eta x)\mathcal{L}_{x,t,\theta,\tau,\eta,q}(y^k)
\end{equation}
and as a result for any polynomial $f\geq0$ (i.e., $f(y)\geq0$ for all $y\in\mathds{R}$) we have
\begin{equation}\label{nieujemnosc}
0\leq\lim\limits_{h\to0^+}\int\limits_{\mathds{R}}f(y)\tfrac{(y-x)^2}{h}\mathds{P}_{t,t+h}(x,\mathrm{d}y)=(1+\eta x)\mathcal{L}_{x,t,\theta,\tau,\eta,q}(f(y))
\end{equation}
From \eqref{martyngal} and \eqref{drugi_moment} the conditional variance for quadratic harnesses is given by
\begin{equation}\label{condvar}
    \var(X_t|\mathcal{F}_s)=(t-s)(1+\eta X_s).
\end{equation}
When $x\in\text{supp}(X_s)$ is such that $1+\eta x=0$, then on the set $\{X_s=x\}$ we have $\var(X_t|\mathcal{F}_s)=0$. Consequently, $x$ is an absorbing state, and the infinitesimal generator is zero according to \eqref{gen_1_wstep}. Consequently, Theorem \ref{twierdzenie_o_generatorze_wiel} holds in this case.\\
Otherwise, if $x\in\text{supp}(X_s)$ then \eqref{condvar} yields $1+\eta x>0$ since the conditional variance is non-negative. Consequently, \eqref{nieujemnosc} implies that $\mathcal{L}_{x,t,\theta,\tau,\eta,q}$ is a nonnegative-definite moment functional, i.e., for all polynomials $f\geq0$ we have 
$\mathcal{L}_{x,t,\theta,\tau,\eta,q}f\geq 0$. The proof of \cite[Theorem 4.4.]{Chihara} implies that the product of consecutive coefficients of $B_{n-1}$ from recurrence \eqref{rekurencja_na_wiel_ort_B} is nonnegative. Therefore by \cite[Theorem A.1.]{bib_BrycWesolowski_2010} there exists a probability measure $\nu_{x,t,\eta,\theta,\tau,q}$ such that for all polynomials $f$ we have
\begin{equation}\label{L_a_mu}
    \mathcal{L}_{x,t,\theta,\tau,\eta,q}\,f=\int\limits_{\mathds{R}}f(y)\;\nu_{x,t,\eta,\theta,\tau,q}(\mathrm{d}y).
\end{equation}
 Putting together \eqref{A_jako_operator} and \eqref{L_a_mu} ends the proof.
\end{proof}

 The measure $\nu_{x,t,\eta,\theta,\tau,q}$ that appears on the right-hand side of \eqref{generator} may not be unique. A problem of uniqueness of the orthogonal measure is equivalent to a Hamburger moment problem with moments respectively equal $\mathcal{L}_{x,t,\theta,\tau,\eta,q}(y^n)$, $n\geq0$, see \cite[p.71]{Chihara}.\\
In addition to the integral representation for the infinitesimal generator, we have shown in the proof above, compare \eqref{y^k} and \eqref{L_a_mu}, that all moments of the measure $\tfrac{(y-x)^2}{h}\mathds{P}_{t,t+h}(x,\mathrm{d}y)$  converge to the respective moments of $(1+\eta x)\nu_{x,t,\eta,\theta,\tau,q}(\mathrm{d}y)$ as $h\to 0^+$. Analogously, we can show the same for  $\tfrac{(y-x)^2}{h}\mathds{P}_{t-h,t}(x,\mathrm{d}y)$. Moreover, we can normalize measures $\tfrac{(y-x)^2}{h}\mathds{P}_{t,t+h}(x,\mathrm{d}y)$ and $\tfrac{(y-x)^2}{h}\mathds{P}_{t-h,t}(x,\mathrm{d}y)$ in order to make them probabilistic.
\begin{remark}\label{miary_probabilistyczne}
	Let $1+\eta x>0$ and  $h>0$. Measures 
 $$\tfrac{(y-x)^2}{h(1+\eta x)}\mathds{P}_{t,t+h}(x,\mathrm{d}y) \qquad \text{and}\qquad \tfrac{(y-x)^2}{h(1+\eta x)}\mathds{P}_{t-h,t}(x,\mathrm{d}y)$$
 are probabilistic.
\end{remark}
\begin{proof}
	Note that, except for the trivial case in \eqref{condvar}, we have that $1+\eta x>0$ for all $x\in\text{supp}(X_t)$.  Then
	$\tfrac{(y-x)^2}{h(1+\eta x)}\mathds{P}_{t,t+h}(x,\mathrm{d}y)$ is non-negative and
	$$\int\limits_\mathds{R}\tfrac{(y-x)^2}{h(1+\eta x)}\mathds{P}_{t,t+h}(x,\mathrm{d}y)=\tfrac{\var(X_{t+h}|X_t=x)}{h(1+\eta x)}=1.$$
	The same arguments can be applied to $\tfrac{(y-x)^2}{h(1+\eta x)}\mathds{P}_{t-h,t}(x,\mathrm{d}y)$.
\end{proof}

From now on, we restrict the domain of $q$ to $q\in[-1,1)$. Then the coefficients in the recurrence \eqref{rekurencja_na_wiel_ort_B} are bounded for fixed $t$ and $x$. Consequently, the polynomials $(B_n(\cdot;x,t))_{n\geq0}$ are orthogonal with respect to the measure with bounded support, see Theorems 2.5.4 and 2.5.5 in \cite{Ismail}. As a result, $\nu_{x,t,\eta,\theta,\tau,q}$ is uniquely determined by its moments, so the convergence of the moments of the measures $\tfrac{(y-x)^2}{h(1+\eta x)}\mathds{P}_{t,t+h}(x,\mathrm{d}y)$ and $\tfrac{(y-x)^2}{h(1+\eta x)}\mathds{P}_{t-h,t}(x,\mathrm{d}y)$ to the moments of $\nu_{x,t,\eta,\theta,\tau,q}$ implies the weak convergence of these measures, see \cite[Theorem 30.2]{billingsley}. As a result, when $q\in[-1,1)$, we can extend the domain of the infinitesimal generator $\mathbf{A}_t$ to bounded continuous functions with bounded continuous second derivative.
\begin{corollary}\label{twierdzenie_o_generatorze_ogolne}
Let $g:\mathds{R}\to\mathds{R}$ be a bounded continuous function with a bounded continuous second derivative. Then for $x\in\mathrm{supp}(X_t)$ we have
    \begin{equation}\label{At+ dla funkcji g}
    \begin{split}
        \mathbf{A}_tg(x)=\tfrac{1+\eta x}{2}g''(&x)\nu_{x,t,\eta,\theta,\tau,q}(\{x\})\\
        &+(1+\eta x)\int\limits_{\mathds{R}\setminus\{x\}}\frac{\partial}{\partial x} \bigg(\frac{g(y)-g(x)}{y-x}\bigg)\nu_{x,t,\eta,\theta,\tau,q}(\mathrm{d}y),
        \end{split}
    \end{equation}
    where $\nu_{x,t,\eta,\theta,\tau,q}$ is the probability measure defined in Theorem \ref{twierdzenie_o_generatorze_wiel}.
\end{corollary}
\begin{proof}
As already observed, if $1+\eta x=0$, then $x$ is an absorbing state and consequently \eqref{At+ dla funkcji g} is trivially satisfied, see \eqref{gen_1_wstep}. So let us assume that $1+\eta x>0$. We will prove \eqref{At+ dla funkcji g} only for $\mathbf{A}^+_t$, because repeating the same argument for $\mathds{P}_{t-h,h}(x,\mathrm{d}y)$ yields the same for $\mathbf{A}^-_t$.\\
Fix $x\in\text{supp}(X_t)$. Let us define a function $\phi_x:\mathbb{R}\to\mathbb{R}$ by the following formula:
$$\phi_x(y):= \left\{ \begin{array}{ll}
\frac{\partial}{\partial x}\tfrac{g(y)-g(x)}{y-x} & \textrm{for $y\ne x$,}\\
\tfrac{1}{2}g''(x) & \textrm{for $y=x$.}
\end{array} \right.$$
By Taylor's theorem, $\phi_x$ is a bounded continuous function since
$$\tfrac{\partial}{\partial x}\tfrac{g(y)-g(x)}{y-x}=\tfrac{1}{(y-x)^2}\int_x^yg''(z)(y-z)\mathrm{d}z\xrightarrow{y\to x}\tfrac{1}{2}g''(x)$$
and
$$\bigg|\tfrac{\partial}{\partial x}\tfrac{g(y)-g(x)}{y-x}\bigg|= \tfrac{1}{(y-x)^2}\bigg|\int_x^yg''(z)(y-z)\mathrm{d}z\bigg|\leq \tfrac{1}{2}\sup_{y\in\mathds{R}}|g''(y)|.$$
Because $g(y)-g(x)=(y-x)g'(x)+\int_x^y\,(y-z)g''(z)\,dz$ we can write
\begin{equation}\label{jhx}
 \int\limits_{\mathds{R}}\tfrac{g(y)-g(x)}{h}\mathds{P}_{t,t+h}(x,\mathrm{d}y)=g'(x)\int\limits_{\mathds{R}}\tfrac{y-x}{h}\mathds{P}_{t,t+h}(x,\mathrm{d}y)+J(h,x),
 \end{equation}
 where
 $$J(h,x)=\int\limits_{\mathds{R}}\tfrac{\int_x^yg''(z)(y-z)\mathrm{d}z}{h}\mathds{P}_{t,t+h}(x,\mathrm{d}y)=\int\limits_{\mathds{R}\setminus\{x\}}\tfrac{\int_x^yg''(z)(y-z)\mathrm{d}z}{(y-x)^2}\cdot\tfrac{(y-x)^2}{h}\mathds{P}_{t,t+h}(x,\mathrm{d}y).$$
 Since $\phi_x(y)=\tfrac{\int_x^yg''(z)(y-z)\mathrm{d}z}{(y-x)^2}$ when $y\neq x$, we get
$$J(h,x)=\int\limits_{\mathds{R}\setminus\{x\}}\phi_x(y)\tfrac{(y-x)^2}{h}\mathds{P}_{t,t+h}(x,\mathrm{d}y)=\int\limits_{\mathds{R}}\phi_x(y)\tfrac{(y-x)^2}{h}\mathds{P}_{t,t+h}(x,\mathrm{d}y).$$
Since the measure $\tfrac{(y-x)^2}{h}\mathds{P}_{t,t+h}(x,\mathrm{d}y)$ converges weakly to $(1+\eta x)\nu_{x,t,\eta,\theta,\tau,q}(\mathrm{d}y)$, we conclude that $\lim\limits_{h\to0^+}J(h,x)=(1+\eta x)\int\limits_{\mathds{R}}\phi_x(y)\nu_{x,t,\eta,\theta,\tau,q}(\mathrm{d}y)$. Consequently, taking the limit when $h\to0^+$ in \eqref{jhx} and using the fact that $\lim\limits_{h\to0^+} \int\limits_{\mathds{R}}\tfrac{y-x}{h}\mathds{P}_{t,t+h}(x,\mathrm{d}y)=\mathbf A_t x=0$, which follows by Theorem \ref{twierdzenie_o_generatorze_wiel},  we obtain the desired result.
\end{proof}
It is worth mentioning that in all known cases  of $QH(\eta,\theta;0,\tau;q)$ processes with $q\in[-1,1)$, the supports of $\mathds{P}_{t,t+h}(x,\mathrm{d}y)$ and $\mathds{P}_{t-h,t}(x,\mathrm{d}y)$ are bounded for any fixed $t, h>0$ and $x\in \mathds{R}$, see \cite{bib_BrycWesolowski_2010}. In such a situation, it is sufficient to assume in Corollary \ref{twierdzenie_o_generatorze_ogolne} that $g$ is a $C^2$ function (we do not need to assume the boundedness). 
\section{More on the measure $\nu_{x,t,\theta,\tau,\eta,q}$}
\label{specjalne wlasnosci}
Note that Theorem \ref{twierdzenie_o_generatorze_wiel} and Corollary \ref{twierdzenie_o_generatorze_ogolne} give explicit formulas for infinitesimal generators in terms of a  measure  $\nu_{x,t}:=\nu_{x,t,\eta,\theta,\tau,q}$. In this section, we embed the polynomials $(B_n(y;x,t))_{n\geq0}$ in the Askey-Wilson scheme, which allows us to describe the measure $\nu_{x,t}$ in terms of measures well-known from the literature. In addition, we relate the measure $\nu_{x,t}$ to the transition probabilities of a  quadratic harness $QH(\eta,\theta;0, \tau;q)$, eventually modifying its parameters slightly. 
\subsection{Case $|q|<1$.}
In this section we assume  \eqref{par_QH} with 
\begin{equation}\label{parametry_dodatkowe}
    \sigma=0\quad\text{and}\quad |q|<1.
\end{equation}
As discussed in the previous section, the probability measures $\tfrac{(y-x)^2}{h(1+\eta x)}\mathds{P}_{t,t+h}(x,\mathrm{d}y)$ and $\tfrac{(y-x)^2}{h(1+\eta x)}\mathds{P}_{t-h,t}(x,\mathrm{d}y)$ converge weakly to the probability measure $\nu_{x,t}$ for all $t>0$. Moreover, the coefficients in the recurrence \eqref{rekurencja_na_wiel_ort_B} at $B_{n-1}$ are bounded, so this measure is uniquely determined. Consequently, for all fixed $t>0$ and $x$ in the support of $X_t$, Favard's theorem implies that the coefficients of $B_{n-1}$ in the three-term recurrence \eqref{rekurencja_na_wiel_ort_B} satisfy 
$$\prod\limits_{n=1}^N\left\{\alpha_t(1+\eta\gamma_t[n]_q+\eta\beta_t[n]_q^2+x\eta q^n)[n+1]_q[n]_q\right\}\geq0$$
for all $N\geq1$. Because $\alpha_t>0$ for $t>0$, the  above condition is equivalent to
\begin{equation}\label{glowny_warunek}
\prod\limits_{n=1}^N(1+\eta\gamma_t[n]_q+\eta\beta_t[n]_q^2+x\eta q^n)\geq0\quad \text{for all }N\geq1;
\end{equation}
note that $q\in(-1,1)$ implies $[n]_q>0$ for all $n\geq1$. According to \eqref{glowny_warunek}, only the following two situations are possible.\\ 
$\bullet$ The first, if there is $n\geq1$ such that
\begin{equation}\label{=0}
    1+\eta\gamma_t[n]_q+\eta\beta_t[n]_q^2+x\eta q^n=0,
\end{equation}
then $\nu_{x,t}$ is supported on distinct zeros of the polynomial $B_{N_0}(\cdot;x,t)$, where $N_0$ is the smallest $n\geq1$ satisfying \eqref{=0}, see \cite[Theorem A.1]{bib_BrycWesolowski_2010}. Moreover, for all $1\leq n< N_0$ we have
\begin{equation*}
    1+\eta\gamma_t[n]_q+\eta\beta_t[n]_q^2+x\eta q^n>0,
\end{equation*}
which imposes additional constraints on the parameters $\eta$, $\theta$, $\tau$, $q$ and on the support of $X_t$ for fixed $t>0$.\\
$\bullet$ The second, if for all $n\geq1$
\begin{equation*}
    1+\eta\gamma_t[n]_q+\eta\beta_t[n]_q^2+x\eta q^n>0.
\end{equation*}
Then taking a limit as $n\to\infty$, we get
\begin{equation}\label{rownanie_ogolne}
    0\leq 1+\tfrac{\eta\gamma_t}{1-q}+\tfrac{\eta\beta_t}{(1-q)^2} =1+\tfrac{\eta\widetilde{\theta}}{1-q},
\end{equation}
where 
\begin{equation}\label{theta_falka}
\widetilde{\theta}:=\theta+\tfrac{\eta\tau}{1-q}.
\end{equation}
Consequently, not all combinations of the parameters of the quadratic harness are possible. Moreover, identity \eqref{glowny_warunek} for $N=1$ implies
\begin{equation}\label{pierwsze_N}
    1+\eta\theta+\eta^2\tau\geq 0.
\end{equation}
Indeed, taking a limit as $t\searrow 0$ and considering a version of the process with c\'adl\'ag trajectories (the quadratic harness has a martingale property, see \eqref{martyngal}, so such a version exists), we get that $x$ tends to zero and consequently \eqref{pierwsze_N} follows.\\
As a result if $\tau=0$, then \eqref{rownanie_ogolne} and \eqref{pierwsze_N} imply that $QH(\eta,\theta;0,0;q)$ exists only when
$$1+\eta\theta\geq\max\{0,q\}.$$
Under the above condition, the construction of the bi-Poisson process is done, see Section \ref{Bi-P} in the appendix.

In general, the measure  $\nu_{x,t,\eta,\theta,\tau,q}$ can be conveniently described as a linear transformation of a measure occurring in the Askey-Wilson scheme.
\begin{theorem}\label{twierdzenie_o_miarach} 
Let us assume \eqref{rownanie_ogolne}. Then for $x\in\mathds{R}$ we have
\begin{equation}\label{miara_ni_na_borelowskich}
  \nu_{x,t,\eta,\theta,\tau,q}(A)=\mu(\{\tfrac{y-w}{u}:y\in A\}),\quad A\in\mathcal{B}(\mathds{R}),  
\end{equation}
where 
\begin{enumerate}
    \item[(i)] for $1-q+\eta\widetilde{\theta}>0$ 
    $$u=-\tfrac{\sqrt{(1-q)^3}}{2\sqrt{\alpha_t}\sqrt{1-q+\eta\widetilde{\theta}}}, \quad w=\tfrac{(1-q)\widetilde{\theta}+\beta_t}{2\sqrt{\alpha_t}\sqrt{1-q}\sqrt{1-q+\eta\widetilde{\theta}}},$$
    and $\mu$ is a orthogonality measure of Askey-Wilson polynomials, see \eqref{rekurencja_Akey_Wilson}, with parameters 
    \small
    \begin{equation}\label{abc}
    \begin{split}
      a&=\tfrac{\eta\sqrt{\alpha_t}}{\sqrt{(1-q)(1-q+\eta\widetilde{\theta})}},\\
      b,c&=\tfrac{1}{2\sqrt{\alpha_t}}\left(q\tfrac{(1-q)\widetilde{\theta}\,+\,\beta_t-(1-q)^2x}{\sqrt{(1-q)(1-q+\eta\widetilde{\theta}})}\pm|q|\sqrt{\tfrac{\big((1-q)\widetilde{\theta}+\beta_t-(1-q)^2x\big)^2}{(1-q)(1-q+\eta\widetilde{\theta})}-4\alpha_t}\right)
      \end{split}
    \end{equation}
    \normalsize
    and $d=0$,
    \item[(ii)] for $1-q+\eta\widetilde{\theta}=0$ and\\
    \begin{enumerate} 
    \item  $q=0$: 
    $u=1$, $w=-\widetilde\theta$ and $\mu=\delta_0$,
    \item $q\ne0$ and $\tfrac{\eta \beta_t}{(1-q)^2}\ne1+x\eta$: $$u=-\tfrac{q\eta(1-q)^2}{\eta\beta_t-(1+x\eta)(1-q)^2}\quad\mbox{and}\quad  w=-u\tfrac{\eta\beta_t-(1-q)^2}{\eta(1-q)^2}$$    and $\mu$ is a orthogonality measure of Big $q$-Jacobi polynomials, see \eqref{rekurencja q-Racah}, with parameters:
    \begin{equation}\label{parametry_q_Racaha}
        a=q,\quad b=0\quad \text{and}\quad  c=\tfrac{\eta\beta_t}{\eta\beta_t-(1+\eta x)(1-q)^2},
    \end{equation}
    \item $q\ne0$ and $\tfrac{\eta \beta_t}{(1-q)^2}=1+x\eta$: $$u=\tfrac{(1-q)^2}{\beta_t}\quad\mbox{and}\quad w=\tfrac{u}{\eta}$$
    and $\mu$ is a orthogonality measure of Little q-Jacobi polynomials with $a=q$ and $b=0$, see \eqref{rekurencja q-little}.
    \end{enumerate}
    \end{enumerate}
\end{theorem}
\begin{proof}
Consider polynomials $(p_n(y))_{n\geq0}$ given by
\begin{equation}\label{p_a_B}
    p_n(y):=u^nB_n(\tfrac{y-w}{u};x,t), \quad n\geq0,
\end{equation} 
where the parameters $u$ and $w$ are given in the formulation of the theorem and generally depend on $x$ and $t$. In view of  \eqref{rownanie_ogolne}, we consider only two cases:
\begin{enumerate}
     \item[(i)] $1-q+\eta\widetilde{\theta}>0$. Then
      $(p_n(y))_{n\geq0}$ in \eqref{p_a_B}  are the Askey-Wilson polynomials with parameters \eqref{abc}. Under the assumed range of parameters (see \eqref{par_QH} and \eqref{parametry_dodatkowe}), $b$ and $c$ are either a complex conjugate pair or both real.
    \item[(ii)] $1-q+\eta\widetilde{\theta}=0$. As a result,  $\eta\ne0$.
     \begin{enumerate} \item  $q=0$: in this case
      we have $p_n(y)=y(y-\beta_t)^{n-1}$, $n\geq1$.
    \item $q\ne0$ and $\tfrac{\eta \beta_t}{(1-q)^2}\ne1+x\eta$: it is easy to verify that  $(p_n(y))_{ n\ge 0}$ in \eqref{p_a_B} are the Big $q$-Jacobi polynomials with parameters \eqref{parametry_q_Racaha}.
    \item $q\ne0$ and $\tfrac{\eta \beta_t}{(1-q)^2}=1+x\eta$: then the polynomials $(p_n(y))_{ n\ge 0}$ are the Little q-Jacobi polynomials with $a=q$ and $b=0$, see \eqref{rekurencja q-little}.
    \end{enumerate}
\end{enumerate}
In the simplest case (ii)(a) we clearly have  $\mu=\delta_{0}$, see \cite[Thereom A.1.]{bib_BrycWesolowski_2010}. In all other cases, the polynomials $(p_n(y))_{n\geq0}$ as well as their orthogonality measures $\mu$ are well known in the literature. Consequently, in view of \eqref{p_a_B}, the measure $\nu_{x,t}$ is identified through \eqref{miara_ni_na_borelowskich}.
\end{proof} 
The measures described in Theorem \ref{twierdzenie_o_miarach} for some parameters may have only an absolutely continuous part, but for other parameters they may also have atoms. Note also that these parameters depend on the support of the process $(X_t)_{t\geq0}$. 

Let us examine in more detail the case when $1-q+\eta\widetilde{\theta}>0$. The parameters $b$ and $c$, see \eqref{abc}, depend on  $x\in\text{supp}(X_t)$. If $x$ belongs to the absolutely continuous part of the support of $X_t$, i.e. to the interval \eqref{przedzial_uw}, then the measure $\nu_{x,t}$ has only the absolutely continuous part.  Indeed, fix $x$ in the interval \eqref{przedzial_uw} for some $t\geq0$. Then we have
$$\left(x-\tfrac{\widetilde{\theta}+ \tfrac{\beta_t}{1-q}}{1-q}\right)^2\leq \tfrac{4\alpha_t(1-q+\eta\widetilde{\theta})}{(1-q)^3}.$$
Therefore, $\tfrac{\big((1-q)\widetilde{\theta}+\beta_t-(1-q)^2x\big)^2}{(1-q)(1-q+\eta\widetilde{\theta})}\leq 4\alpha_t$, so the expression under a square root in \eqref{abc} is negative, hence $b$ and $c$ are complex conjugate and
$$a>0,\quad |b|^2=|c|^2=q^2<1.$$
Consequently, using the notation  from Section \ref{Askey-Wilson polynomials}, we have $m_1=m_2=0$, so the Askey-Wilson distribution $\mu$ exists and has only a continuous component.\\
Now let us assume 
\begin{equation}\label{parametry}
1+\eta\widetilde{\theta}\geq \max\{0,q\},
\end{equation}
which is slightly stronger than \eqref{rownanie_ogolne}.
For $\tau=0$, condition \eqref{parametry} is not so much restrictive since \eqref{pierwsze_N} holds. For $\tau>0$, conditions \eqref{rownanie_ogolne} and \eqref{pierwsze_N} are incomparable with \eqref{parametry}, especially when $q<0$.\\
It appears that, under the assumption \eqref{parametry}, the measure $\nu_{x,t}$ can be related to transition probabilities of the bi-Poisson process described in detail in the Appendix, see in particular Section \ref{Bi-P}.
\begin{theorem}\label{twierdzenie_o_mu}
Let us assume \eqref{parametry} and consider a bi-Poisson process $QH(\eta,\widetilde{\theta};0,0;q)$ with transition probabilities $\mathds{Q}_{s,t}(x,\mathrm{d} y)$, $0\le s<t$. Then for all $t>0$ and $x\in \mathds{R}$ we have
$$\nu_{x,t,\eta,\theta,\tau,q}(\mathrm{d} y)=\mathds{Q}_{\tfrac{q^2\alpha_t}{1-q}, \tfrac{\alpha_t}{1-q}}(qx+\gamma_t+\beta_t,\mathrm{d} y).$$ 
\end{theorem}
\begin{proof}
Under the assumed range of parameters. $QH(\eta,\widetilde{\theta};0,0;q)$ is well-defined. Then comparing proper recurrences gives 
$$B_n(y;x,t)=Q_n(y;q x+\gamma_t+\beta_t,\tfrac{\alpha_t}{1-q},\tfrac{q^2\alpha_t}{1-q})=\tfrac{Q_{n+1}\left(y;x,\tfrac{\alpha_t}{1-q},\tfrac{\alpha_t}{1-q}\right)}{y-x}$$ 
for all $n\geq0$, $t\geq0$ and $x\in\mathds{R}$, where $(Q_n(y;x,t,s))_{n\geq0}$ satisfy  \eqref{rekurencja_na_wiel_Q} with $\theta$ replaced by $\widetilde{\theta}$ (since we consider $QH(\eta,\widetilde{\theta};0,0,q)$). Consequently $(B_n(y;x,t))_{n\geq 0}$ are orthogonal with respect to the probability measure $\mathds{Q}_{\tfrac{q^2\alpha_t}{1-q}, \tfrac{\alpha_t}{1-q}}(qx+\gamma_t+\beta_t,\mathrm{d} y)$ 
for all $x\in\mathds{R}$.
\end{proof}
\begin{remark}
    Since distributions of quadratic harnesses are obtained as linear transformations of Askey-Wilson distributions, the assumptions of Proposition 1.2 in \cite{bridges} are satisfied, and $QH(\eta,\theta;0,\tau;q)$ can be obtained as a one-sided bridge in the bi-Poisson process $QH(\eta,\widetilde{\theta};0,0;q)$, see \cite[Remark 2.4, (ii)]{bridges} with $r:=\tfrac{\tau}{1-q}$ and $z_r:=0$ if only zero is in the support of $QH(\eta,\widetilde{\theta};0,0;q)$ at time $r$. Particularly, the interval given in \eqref{przedzial} for $s=\tfrac{\tau}{1-q}$ is in that support, so it is easy to verify that $0$ is in this interval if and only if $\theta^2\leq 4\tau$. Moreover, zero is also trivially in the support at time $r$ if $\tau=0$.\\
    As a result in these cases the transition probabilities $\mathds P_{s,t}$ and $\mathds{Q}_{\tilde s,\tilde t}$ of $QH(\eta,\theta;0,\tau;q)$ and $QH(\eta,\widetilde{\theta};0,0;q)$, respectively, satisfy     
    \begin{equation}\label{P_a_Q}
    \mathds P_{s,t}(x,\mathrm{d} y)=
        \mathds{Q}_{\tfrac{\alpha_s}{1-q},\tfrac{\alpha_t}{1-q}}(x,\mathrm{d} y), \qquad 0\leq s<t, \quad x\in\mathds{R}
    \end{equation}
    and in view of Theorem \ref{twierdzenie_o_mu} we additionally have 
    $$\nu_{x,t,\eta,\theta,\tau,q}(\mathrm{d} y)=\mathds{P}_{q^2t-(1+q)\tau,t}(qx+\gamma_t+\beta_t,\mathrm{d} y)$$ 
    when $q^2t\geq(1+q)\tau$.
\end{remark}
\begin{remark}
The relation \eqref{P_a_Q} can  also be obtained by comparing the recurrence for $(Q_n(y;x,\tfrac{\alpha_t}{1-q},\tfrac{\alpha_s}{1-q}))_{n\geq0}$ satisfying  \eqref{rekurencja_na_wiel_Q} with $\theta$ replaced by $\widetilde{\theta}$ with the recurrence for $(\widetilde{Q}_n(y;x,t,s))_{n\geq0}$ satisfying \eqref{rekurencja_na_wiel_Q_falka}. The uniqueness of the orthogonality measure implies the statement.
\end{remark}

According to Theorems \ref{twierdzenie_o_generatorze_wiel} and \ref{twierdzenie_o_mu}, the infinitesimal generator  of $QH(\eta,\theta;0,\tau;q)$ (whose parameters satisfy additional assumptions, including \eqref{parametry}) acting on polynomials is given by
\begin{equation*}
\mathbf{A}_t(f)(x)=(1+\eta x)\int\limits_{\mathds{R}}\frac{\partial}{\partial x} \frac{f(y)-f(x)}{y-x}\mathds{P}_{q^2t-(1-q)\tau,t}(qx+\gamma_t+\beta_t,\mathrm{d}y)
\end{equation*}
for all $x\in\text{supp}(X_t)$. The infinitesimal generator $\mathbf{A}_t$ maps polynomials to polynomials in the variable $x$ on the suitably large set containing the interval \eqref{przedzial_uw}, so the formula for $\mathbf{A}_t$ is extendable for all $x\in\mathds{R}$.

In addition, we can also represent the infinitesimal generator using the transition function of another quadratic harness.

\begin{theorem}
If $QH(\eta,\theta+(1+q)\eta\tau;0,q^2\tau;q)$  with transition probabilities $\mathds{\widetilde{P}}_{s,t}(x,\mathrm{d}y)$ exists, then for all $t\geq0$ and $x\in\mathds{R}$
$$\nu_{x,t,\eta,\theta,\tau,q}(\mathrm{d} y)=\mathds{\widetilde{P}}_{q^2t,t+\tau(1+q)}(qx-q \eta t+\theta+\eta\tau,\mathrm{d} y).$$
\end{theorem}
\begin{proof}
Because the polynomials $(B_n(y;x,t))_{n\geq0}$ given in \eqref{rekurencja_na_wiel_ort_B} can be written in terms of the polynomials $(\widetilde{Q}_n(y;x,t,s))_{n\geq0}$ given in \eqref{rekurencja_na_wiel_Q_falka} with $\tau$ replaced by $q^2\tau$ and $\theta$ by $\theta+(1+q)\eta\tau$, respectively (since we are considering such a quadratic harness). Namely, the following identity is satisfied
$$B_n(y;x,t)=\widetilde{Q}_n(y;qx-q \eta t+\theta+\eta\tau,t+\tau(1+q),q^2t), \quad n\geq0.$$
Then the thesis holds by the uniqueness of the orthogonality measure of $(B_n(y;x,t))_{n\geq 0}$.
\end{proof}
To conclude this subsection, we derive the exact formula for the infinitesimal generator in the special case of a free quadratic harness.
\begin{example} 
When $q=0$, the recurrence \eqref{rekurencja_na_wiel_ort_B} assumes the form
\begin{equation*}
\begin{split}
B_0(y;x,t)=1, \quad \quad B_1(y;&x,t)=y-\theta-\eta\tau,\\
yB_{n}(y;x,t)=B_{n+1}(y;x,t)&+(\theta+\eta\tau-\eta t)B_{n}(y;x,t)\\
&+(\tau+t)\big(1+\eta\theta+\eta^2\tau\big)B_{n-1}(y;x,t) \quad\text{ for } n\geq 1.
\end{split}
\end{equation*}
It is easy to check that the polynomials $(B_n(y;x,t))_{n\geq0}$ are orthogonal with respect to a probability measure if and only if $1+\eta\theta+\eta^2\tau\geq0$. Therefore, if $1+\eta\theta+\eta^2\tau<0$, then the quadratic harness $QH(\eta,\theta;0,\tau;0)$ does not exist.\\
When $1+\eta\theta+\eta^2\tau=0$, then  $(B_n(y;x,t))_{n\geq0}$ are orthogonal with respect to a Dirac measure concentrated in the point $\theta+\eta\tau$. Polynomials $(B_n(y;x,t))_{n\geq 0}$ do not depend on $x$, so the measure does not depend on $x$ as well.\\
When $1+\eta\theta+\eta^2\tau>0$, comparing respective recurrences, we conclude that 
$$B_n(y;x,t)=(2\sigma_t)^nP^*_n\left(\tfrac{y-m_t}{2\sigma_t},\tfrac{\eta  t}{2 \sigma_t}\right), \quad n\geq0,$$
where $m_t:=\theta+\eta\tau-\eta t$, $\sigma_t^2:=(\tau+t)\big(1+\eta\theta+\eta^2\tau\big)$ and $(P_n^*(y,c))_{n\geq0}$ are the polynomials the same as in Example $(a)$ from \cite[Section 5]{chihara_art}. Consequently,  if $|\eta t|\leq \sigma_t$, then $(B_n(y;x,t))_{n\geq0}$ are orthogonal with respect to 
$$\nu_{x,t}(\mathrm{d}y)=\tfrac{1}{2\pi}\tfrac{\sqrt{4\sigma_t^2-(y-m_t)^2}}{\eta^2 t^2+\sigma_t^2-\eta t(y-m_t)}\mathds{1}_{(m_t-2\sigma_t,m_t+2\sigma_t)}(y)\,\mathrm{d}y.$$
Therefore,
\begin{equation*}
\mathbf{A}_tf(x)=\frac{1+\eta x}{2\pi}\int\limits_{m_t-2\sigma_t}
^{m_t+2\sigma_t}\frac{\partial}{\partial x} \frac{f(y)-f(x)}{y-x}\tfrac{\sqrt{4\sigma_t^2-(y-m_t)^2}}{\eta^2 t^2+\sigma_t^2-\eta t(y-m_t)}\mathrm{d} y
\end{equation*}
for $f$ being a polynomial or a continuous function with a continuous second derivative (in this case,  the supports of the transition probabilities are compact, so  $f$ is bounded on the support). 
In particular, when $\eta=0$, the measure $\nu_{x,t}$ is the Wigner's semicircle law with the mean $\theta$ and the variance $\tau+t$. If $|\eta t|> \sigma_t$, then $(B_n(y;x,t))_{n\geq0}$ are orthogonal with respect to 
$$\nu_{x,t}(\mathrm{d}y)=\tfrac{1}{2\pi}\tfrac{\sqrt{4 \sigma_t^2-(y-m_t)^2}}{\eta t(u_t-y)}\mathds{1}_{(m_t-2\sigma_t,m_t+2\sigma_t)}(y)\mathrm{d}y+\left(1 - \tfrac{\sigma_t^2}{\eta^2t^2}\right)\delta_{u_t}(\mathrm{d}y),$$
where $u_t:=\eta t+m_t+\tfrac{\sigma_t^2}{\eta t}$. Thus, in this case, we need to include an additional summand in the formula for the infinitesimal generator arising from the atom of $\nu_{x,t}$.
\end{example}
\subsection{Case q=-1.} In this case the polynomials $(B_n(y;x,t))_{n\geq0}$ are orthogonal with respect to a Dirac measure $\delta_{\theta+\eta(t+\tau)-x}$ at the point $\theta+\eta(t+\tau)-x$, because the coefficient at $B_{n-2}$ in the recurrence \eqref{rekurencja_na_wiel_ort_B} vanishes ($[2n]_q=0$ for all $n\geq0$). Therefore, Theorem \ref{twierdzenie_o_generatorze_wiel} and Corollary \ref{twierdzenie_o_generatorze_ogolne} imply that the domain of the infinitesimal generator contains polynomials and bounded  continuous functions with bounded continuous second derivative $g$, and it takes a form:
\small
\begin{equation}\label{generator_q=-1}
	 \mathbf{A}_t(g)(x)= \left\{ \begin{array}{ll}
	 	\tfrac{1+\eta x}{2}g''(x) & \textrm{when $\theta+\eta(t+\tau)=2x$,}\\
	 	\tfrac{1+\eta x}{\theta+\eta(t+\tau)-2x}\left(\tfrac{g(\theta+\eta(t+\tau)-x)-g(x)}{\theta+\eta(t+\tau)-2x}-g'(x)\right)&  \textrm{when $\theta+\eta(t+\tau)\ne 2x$.}
	 \end{array} \right.
\end{equation}
\normalsize
Furthermore, the construction of bi-Poisson process $QH(\eta,\theta;0,0;-1)$ was presented when $1+\eta\theta\geq0$ in \cite[Subsection 3.2.]{bib_BrycMatysiakWesolowski_1}. Especially, there exists  $QH(\eta,\theta+\eta\tau;0,0;-1)$ with the parameters satisfying  \eqref{pierwsze_N}.\\
Surprisingly, the tedious calculations show that this process satisfies \eqref{def_harness} and \eqref{def_QH}. Therefore, $QH(\eta,\theta+\eta\tau;0,0;-1)$ is also $QH(\eta,\theta;0,\tau;-1)$.\\
Then, by reading out the transition probabilities, see  \cite[Subsection 3.2.]{bib_BrycMatysiakWesolowski_1}, it is easy to compute the infinitesimal generator directly and obtain exactly the formula from the second line in \eqref{generator_q=-1}. 

\subsection{Case q=-1.} In this case the polynomials $(B_n(y;x,t))_{n\geq0}$ are orthogonal with respect to a Dirac measure $\delta_{\theta+\eta(t+\tau)-x}$ at the point $\theta+\eta(t+\tau)-x$, because the coefficient at $B_{n-2}$ in the recurrence \eqref{rekurencja_na_wiel_ort_B} vanishes ($[2n]_q=0$ for all $n\geq0$). Therefore, Theorem \ref{twierdzenie_o_generatorze_wiel} and Corollary \ref{twierdzenie_o_generatorze_ogolne} imply that the domain of the infinitesimal generator contains polynomials and bounded  continuous functions with bounded continuous second derivative $g$, and it takes a form:
\small
\begin{equation}\label{generator_q=-1}
	 \mathbf{A}_t(g)(x)= \left\{ \begin{array}{ll}
	 	\tfrac{1+\eta x}{2}g''(x) & \textrm{when $\theta+\eta(t+\tau)=2x$,}\\
	 	\tfrac{1+\eta x}{\theta+\eta(t+\tau)-2x}\left(\tfrac{g(\theta+\eta(t+\tau)-x)-g(x)}{\theta+\eta(t+\tau)-2x}-g'(x)\right)&  \textrm{when $\theta+\eta(t+\tau)\ne 2x$.}
	 \end{array} \right.
\end{equation}
\normalsize
Furthermore, the construction of bi-Poisson process $QH(\eta,\theta;0,0;-1)$ was presented when $1+\eta\theta\geq0$ in \cite[Subsection 3.2.]{bib_BrycMatysiakWesolowski_1}. Especially, there exists  $QH(\eta,\theta+\eta\tau;0,0;-1)$ with the parameters satisfying  \eqref{pierwsze_N}.\\
Surprisingly, the tedious calculations show that this process satisfies \eqref{def_harness} and \eqref{def_QH}. Therefore, $QH(\eta,\theta+\eta\tau;0,0;-1)$ is also $QH(\eta,\theta;0,\tau;-1)$.\\
Then, by reading out the transition probabilities, see  \cite[Subsection 3.2.]{bib_BrycMatysiakWesolowski_1}, it is easy to compute the infinitesimal generator directly and obtain exactly the formula from the second line in \eqref{generator_q=-1}.

%%%%%%%%%%%%%%%%%%%%%%%%%%%%%%%%%%%%%%%%%%%%%%%%%%%%%%%%%%%%%%%%%%%
%%                                                               %%
%% Supplementary Material, if any, should be provided in         %%
%% {supplement} environment  with title and short description.   %%
%%                                                               %%
%%%%%%%%%%%%%%%%%%%%%%%%%%%%%%%%%%%%%%%%%%%%%%%%%%%%%%%%%%%%%%%%%%%

\appendix
\section{The Askey scheme}
\setcounter{subsection}{0}
\setcounter{equation}{0}
\setcounter{theorem}{0}
\renewcommand{\theequation}{A.\arabic{equation}}% with dot
\renewcommand\thesubsection{A.\arabic{subsection}}
In this section we recall families of Askey-Wilson polynomials used in Section \ref{specjalne wlasnosci}. We pay special attention to their orthogonality measures.
\subsection{Askey-Wilson polynomials}\label{Askey-Wilson polynomials}
Let us define for $a,b,c,d\in\mathds{C}$  such that
\begin{equation}\label{warunkinaparametry}
    abcd, qabcd \not\in [1,\infty)
\end{equation}
and $|q|<1$ polynomials $(P_n(y))_{n\geq0}$ by a recurrence
$$2y P_n(y)=\widetilde{A}_nP_{n+1}(y)+B_nP_n(y)+\widetilde{C}_nP_{n-1}(y), \quad n\geq0,$$
with the initial conditions $P_{-1}\equiv 0$ and $P_0\equiv 1$,
where for $n\geq0$ we have
\begin{equation*}
    \begin{split}
    \widetilde{A}_n&:=\tfrac{A_n}{(1-ab q^n)(1-ac q^n)(1-ad q^n)},\\
    B_n&:=a+\tfrac{1}{a}-\tfrac{A_n}{a}-aC_n,\\
    \widetilde{C}_n&:=C_n(1-ab q^{n-1})(1-ac q^{n-1})(1-ad q^{n-1}),\\
    A_n&:=\tfrac{(1-abq^n)(1-acq^n)(1-adq^n)(1-abcdq^{n-1})}{(1-abcd q^{2n-1})(1-abcd q^{2n})},\\
    C_n&:=\tfrac{(1-q^n)(1-bcq^n)(1-bdq^n)(1-cdq^{n-1})}{(1-abcd q^{2n-2})(1-abcd q^{2n-1})}.
    \end{split}
\end{equation*}
Here $A_0$ and $C_0$ should be interpreted as $\tfrac{(1-ab)(1-ac)(1-ad)}{1-abcd}$ and $0$, respectively. Then polynomials $(P_n(y))_{n\geq0}$ are Askey-Wilson polynomials, see \cite[Section 14.1]{Koekoek}. Because the coefficients $\widetilde{A}_n$, $B_n$, $\widetilde{C}_n$ do not depend on the order of the parameters $a,b,c,d$, these polynomials are well defined also in the case when $a=0$.\\
Moreover, we can normalize the polynomials $(P_n(y))_{n\geq0}$ by a formula
$$p_n(y):=2^n\prod\limits_{k=0}^{n-1}\widetilde{A}_kP_n(y),\quad  n\geq0,$$
with convention that $\widetilde{A}_{-1}:=1$, to obtain that the polynomials $(p_n(y))_{n\geq0}$ satisfy
\begin{equation}\label{rekurencja_Akey_Wilson}
   y p_n(y)=p_{n+1}(y)+\tfrac{1}{2}B_np_n(y)+\tfrac{1}{4}A_{n-1}C_np_{n-1}(y), \quad n\geq0, 
\end{equation}
with $A_{-1}=1$.\\
Polynomials $(p_n(y))_{n\geq0}$ satisfy three-step recurrence, so there exists a moment functional that makes them orthogonal. However, it is difficult to give explicit conditions in terms of $a$, $b$, $c$, and $d$ when the orthogonality measure $\mu_{a,b,c,d}(\mathrm{d}y)$ for the moment functional exists. It is known only in some special cases, so let us present the results covering all our problems from Section \ref{specjalne wlasnosci}.\\
Denote
\begin{equation}\label{m1,m2}
    \begin{split}
    m_1&:=\sharp\big(\{ab, ac, ad, bc, bd, cd\}\cap [1,\infty)\big),\\
    m_2&:=\sharp\big(\{qab, qac, qad, qbc, qbd, qcd\}\cap [1,\infty)\big).
    \end{split}
\end{equation}
According to \cite[Lemma 3.1]{bib_BrycWesolowski_2010} when $a$, $b$, $c$, $d$ are either real or come in complex conjugate pairs and  satisfy \eqref{warunkinaparametry}, then the distribution $\mu_{a,b,c,d}$  exists only in the following cases:
\begin{enumerate}
    \item If $q\geq 0$ and $m_1=0$, then $\mu_{a,b,c,d}$ has only a continuous component.
    \item If $q < 0$ and $m_1=m_2=0$, then $\mu_{a,b,c,d}$  has only a continuous component.
    \item If $q\geq 0$ and $m_1 = 2$, then $\mu_{a,b,c,d}$  is well-defined if either $q = 0$ or the smaller of the two products that fall into $[1,\infty)$ is of the form $\tfrac{1}{q^N}$ , and in this latter case $\mu_{a,b,c,d}$  is a purely discrete measure with $N + 1$ atoms.
    \item If $q < 0$ and $m_1 = 2$, $m_2 = 0$, then $\mu_{a,b,c,d}$  is well defined if the smaller of the two products in $[1,\infty)$ equals $\tfrac{1}{q^N}$ with even $N$. Then $\mu_{a,b,c,d}$  is a purely discrete measure with $N + 1$ atoms.
    \item If $q < 0$, $m_1 = 0$ and $m_2 = 2$, then $\mu_{a,b,c,d}$ is well defined if the smaller of the two products in $[1,\infty)$ equals $\tfrac{1}{q^N}$ with even $N$. Then $\mu_{a,b,c,d}$  is a purely discrete measure with $N + 2$ atoms.
\end{enumerate}
Introducing for $w,w_1,\ldots,w_k\in\mathds{C}$ the following notation:
\begin{equation}\label{notacja}
    (w;q)_n:= \left\{ \begin{array}{ll}
1 & \textrm{when $n=0$},\\
\prod\limits_{j=0}^{n-1} (1-wq^j)& \textrm{when $n=1,2,\ldots$},\\
\prod\limits_{j=0}^{\infty} (1-wq^j)& \textrm{when $n=\infty$}
\end{array} \right.
\end{equation}
and 
$$(w_1,w_2,\ldots,w_k;q)_n:=(w_1;q)_n\cdot(w_2;q)_n\cdot\ldots\cdot(w_k;q)_n,$$
the probability measure $\mu_{a,b,c,d}$ can be written explicitly as
\begin{equation*}
    \mu_{a,b,c,d}(\mathrm{d}y)= f_{a,b,c,d}(y)\mathds{1}_{\{|y|<1\}}\mathrm{d}y + \sum\limits_{x\in F_{a,b,c,d}}\rho(x)\delta_x(\mathrm{d}y),
\end{equation*}
where for $\theta$ such that $y=\cos(\theta)$ we have
$$f_{a,b,c,d}(y):= \tfrac{(q, ab, ac, ad, bc, bd, cd;q)_\infty}{2\pi(abcd;q)_\infty\sqrt{1-y^2}}\left|\tfrac{(e^{2i\theta};q)_{\infty}}{(ae^{2i\theta},be^{2i\theta},ce^{2i\theta},de^{2i\theta};q)_{\infty}}\right|^2,$$
and $F_{a,b,c,d}$ is an empty or finite set of atoms that arise from each of the parameters $a$, $b$, $c$, $d$ with an absolute value larger than one.
For example, if $a\in(-\infty,-1)\cup(1,\infty)$, then the corresponding atoms are equal
$$x_k=\tfrac{aq^k+(aq^k)^{-1}}{2}$$
for $k=0,1,\ldots$ such that $|aq^k|>1$.
The probabilities of $x_k$ are then equal
\begin{equation*}
    \begin{split}
        \rho(x_0)&:=\tfrac{\left(\tfrac{1}{a^2}, bc, bd, cd;q\right)_{\infty}}{\left(\tfrac{b}{a}, \tfrac{c}{a}, \tfrac{d}{a}, abcd;q\right)_{\infty}},\\
        \rho(x_k)&:= \rho(x_0)\tfrac{\left(a^2, ab, ac, ad;q\right)_k(1-a^2q^{2k} )}{\left(q, \tfrac{qa}{b},\tfrac{qa}{c}, \tfrac{qa}{d};q\right)_k(1-a^2)} 
        \left(\tfrac{q}{abcd}\right)^k,  \quad k\geq1.
    \end{split}
\end{equation*}
The above formula must be rewritten when $abcd=0$. Especially, when $d=0$ we have
\begin{equation*}
        \rho(x_k)= \rho(x_0)\frac{\left(a^2, ab, ac;q\right)_k(1-a^2q^{2k})}{\left(q, \tfrac{qa}{b},\tfrac{qa}{c};q\right)_k(1-a^2)} (-1)^k q^{-\binom{k}{2}}
        \left(\tfrac{1}{a^2bc}\right)^k,  \quad k\geq1,
\end{equation*}
where by convention we put $\binom{1}{2}=0$.
\subsection{Big q-Jacobi polynomials}\label{q-Racah polynomials}
Let us consider  polynomials $(p_n(x))_{n\geq0}$ given by a recurrence
\begin{equation}\label{rekurencja q-Racah}
    x p_n(x)=p_{n+1}(x)+(1-(A_n+C_n))p_n(x)+A_{n-1}C_np_{n-1}(x), \quad n\geq0,
\end{equation}
with the initial conditions $p_{-1}\equiv 0$ and $p_0\equiv 1$,
where for $n\geq0$ we have
\begin{equation*}
    \begin{split}
    A_n&:=\tfrac{(1-a q^{n+1})(1-ab q^{n+1})(1-c q^{n+1})}{(1-ab q^{2n+1})(1-ab q^{2n+2})},\\
    C_n&:=-acq^{n+1}\tfrac{(1-q^n)(1-abc^{-1} q^n)(1-b q^n)}{(1-ab q^{2n})(1-ab q^{2n+1})}.
    \end{split}
\end{equation*}
Then polynomials $(p_n(y))_{n\geq0}$ are normalized Big q-Jacobi polynomials, see \cite[Section 14.5]{Koekoek}.\\
The orthogonality relation for $0 < aq < 1$, $0 \leq bq < 1$ and $c < 0$ is
\begin{equation*}
    \int_{cq}^{aq}w(x)p_m(x)p_n(x)\mathrm{d}_q(x)=h_n\delta_{mn},
\end{equation*}
where
$$w(x):=\tfrac{(a^{-1}x,c^{-1}x;q)_\infty}{(x,bc^{-1}x;q)_\infty},$$
$$h_n=aq(1-q)
\tfrac{(q,abq^2,a^{-1}c,ac^{-1}q;q)_\infty}{
(aq,bq,cq,abc^{-1}q;q)_\infty}
\tfrac{(1-abq)}{(1-abq^{2n+1})}
\tfrac{(q,aq,bq,cq,abc^{-1}q;q)_n}{(abq,abq^{n+1},abq^{n+1};q)_n}
(-acq^2)^nq^{\binom{n}{2}}$$
and
\begin{equation}\label{calka_dq}
    \int_{cq}^{aq}f(x)\mathrm{d}_q(x):=aq(1-q)\sum\limits_{k=0}^\infty f(aq^{k+1})q^k -cq(1-q)\sum\limits_{k=0}^\infty f(cq^{k+1})q^k.
\end{equation}
Above, we used the notation introduced in \eqref{notacja}. More information about Big q-Jacobi polynomials can be found in \cite[Section 14.5]{Koekoek}. In particular, it turns out that they can be obtained as a limit of specially reparameterized Askey-Wilson polynomials.
\subsection{Little q-Jacobi polynomials} Substituting $c qx$ instead of $x$ in $(p_n(x))_{n\geq0}$ given in \eqref{rekurencja q-Racah} and going with $c\to-\infty$ leads to the little q-Jacobi polynomials which, after normalization, satisfy the following recurrence:
\begin{equation}\label{rekurencja q-little}
    x w_n(x)=w_{n+1}(x)+(\widetilde{A}_n+\widetilde{C}_n)w_n(x)+\widetilde{A}_{n-1}\widetilde{C}_nw_{n-1}(x), \quad n\geq0,
\end{equation}
with $w_{-1}\equiv 0$ and $w_{0}\equiv 1$,
where
\begin{equation*}
    \begin{split}
        \widetilde{A}_n&:=q^n\tfrac{(1-a q^{n+1})(1-ab q^{n+1})}{(1-abq^{2n+1})(1-abq^{2n+2})},\\
        \widetilde{C}_n&:=aq^n\tfrac{(1-q^n)(1-bq^n)}{(1-abq^{2n})(1-abq^{2n+1})}.
    \end{split}
\end{equation*}
In this case, the orthogonality relation takes a form:
$$\sum\limits_{k=0}^\infty\tfrac{(bq;q)_k}{(q,q)_k}(aq)^kp_n(q^k)p_m(q^k)=\tfrac{(abq^2;q)_\infty}{(ab;q)_\infty}\tfrac{(1-abq)(aq^n)^n}{(1-abq^{2n+1})}\tfrac{(q,aq,bq;q)_n}{(abq,abq^{n+1},abq^{n+1};q)_n}\delta_{nm}$$
for $0<aq<1$ and $bq<1$. For more information on the Little q-Jacobi polynomials, see \cite[Section 14.12]{Koekoek}.
\subsection{Al-Salam-Carlitz I polynomials}\label{Al-Salam-Carlitz}
Let $a\in\mathds{R}$. We consider polynomials $(p_n(y))_{n\geq0}$ given by the following three-step recurrence:
\begin{equation}\label{rekurencja_Al-Salam-Carlitz}
    xp_n(x) = p_{n+1}(x)+(a+1)q^np_n(x)-aq^{n-1}(1-q^n)p_{n-1}(x), \quad n\geq0,
\end{equation}
with $p_{-1}\equiv0$ and $p_0\equiv 1$. Polynomials $(p_n(y))_{n\geq0}$ are called Al-Salam-Carlitz I polynomials, see \cite[Section 14.24]{Koekoek}.\\
For $a<0$, these polynomials are orthogonal and satisfy
\begin{equation*}
   \int\limits_a^1(qx,a^{-1}qx;q)_\infty p_n(x)p_m(x)\mathrm{d}_q(x)=(-a)^n(1-q)(q;q)_n(q,a,a^{-1}q;q)_\infty q^{{n\choose 2}}\delta_{mn},
\end{equation*}
where we used the notation introduced in \eqref{notacja} and \eqref{calka_dq}.

\section{Construction of quadratic harness $QH(\eta,\theta;0,\tau;q)$}
\setcounter{subsection}{0}
\setcounter{equation}{0}
\setcounter{theorem}{0}
\renewcommand{\theequation}{B.\arabic{equation}}% with dot
\renewcommand\thesubsection{B.\arabic{subsection}}
\renewcommand{\thetheorem}{B.\arabic{theorem}}% with dot
In this section we recall the construction of the bi-Poisson process and complete the missing parts of the proof from \cite{bib_BrycMatysiakWesolowski_1}. In addition, we derive three-term recurrence for orthogonal polynomials, whose orthogonality measure is the transition probability of $QH(\eta,\theta;0,\tau;q)$.
\subsection{Bi-Poisson process}\label{Bi-P}
Let $|q|<1$. The bi-Poisson process is a quadratic harness $QH(\eta,\theta;0,0;q)$. According to  \cite{bib_BrycMatysiakWesolowski_1}, the bi-Poisson process is well-defined under the  constraints:
\begin{equation}\label{zalozenia_parametry}
    1+\eta\theta\geq\max\{q,0\}.
\end{equation}
Condition \eqref{zalozenia_parametry} is sufficient for the existence of this process but not involves all possible combinations of parameters. Moreover, the proof given in \cite{bib_BrycMatysiakWesolowski_1}  covers only the case $1+\eta\theta>\max\{q,0\}$. 
The construction of the process was based on orthogonal polynomials ${Q_n(y;x,t,s)}_{n\geq0}$ given by the following three-step recurrence:
\begin{equation}
\label{rekurencja_na_wiel_Q}
\begin{split}
yQ_n(y;x,t,s)&=Q_{n+1}(y;x,t,s)+\mathcal{A}_{n}(x,t,s)Q_n(y;x,t,s)\\
&+\mathcal{B}_{n}(x,t,s)Q_{n-1}(y;x,t,s),\quad n\ge 0
\end{split}
\end{equation}
with $Q_{-1}(y;x,t,s)=0$, $Q_{0}(y;x,t,s)=1$ and the coefficients $\mathcal{A}_0(x,t,s)=x$, $\mathcal{B}_0(x,t,s)=0$, and
\begin{equation*}
\begin{split}
\mathcal{A}_{n}(x,t,s)&=q^nx+[n]_q(\eta t +\theta-\eta s(1+q)q^{n-1}),\\
\mathcal{B}_{n}(x,t,s)&=[n]_q(t-sq^{n-1})\{1+\eta x q^{n-1}+\eta[n-1]_q(\theta-\eta s q^{n-1})\},\qquad n\ge 1.
\end{split}
\end{equation*}	
The transition probabilities $\mathds{Q}_{s,t}(x,\mathrm{d}y)$, $0\leq s<t$, are defined as unique probabilistic orthogonality measures of polynomials $(Q_n(y;x,t,s))_{n\geq0}$ for
$$x\in \mathcal{U}_s:=\bigcap_{n=1}^\infty\bigg\{y\in\mathds{R}: \prod_{k=1}^{n} \mathcal{B}_k(y,t,s)\geq0\bigg\}.$$
Note that $\mathcal{U}_s$ does not depend on $t$, because $t>s\geq0$ and $|q|<1$. In the result, we can rewrite $\mathcal{U}_s$ as
\begin{equation*}
\mathcal{U}_s=\bigcap_{n=1}^\infty\bigg\{y\in\mathds{R}: \prod_{k=1}^{n} \mathcal{C}_k(y,s)\geq0\bigg\}
\end{equation*}
with
\begin{equation*}
    \mathcal{C}_n(y,s):=1+\eta yq^{n-1}+\eta[n-1]_q(\theta-\eta sq^{n-1}), \quad n\geq1.
\end{equation*}
Further analysis in \cite{bib_BrycMatysiakWesolowski_1} required the sharp inequality $1+\eta \theta>\max\{0,q\}$. Then it was proved that the support of $X_s$, $s>0$, consists of an interval
\begin{equation}\label{przedzial}
\bigg[\tfrac{\theta+\eta s-2\sqrt{s}\sqrt{\eta\theta+1-q}}{1-q},\tfrac{\theta+\eta s+2\sqrt{s}\sqrt{\eta\theta+1-q}}{1-q}\bigg]    
\end{equation}
and possibly a finite set of discrete points:
\begin{itemize}
    \item for $0<s<\tfrac{\theta^2}{\eta\theta+1-q}$ 
    \begin{equation}\label{x_k_pierwszy}
        x_k=-\tfrac{1}{1-q}\left(\theta q^k+s\tfrac{\eta\theta+1-q}{\theta q^k}-(\theta+\eta s)\right), 
    \end{equation}
    where $k=0,1,\ldots \text{ such that } s(\eta\theta+1-q)<q^{2k}\theta^2$,
    \item for $\eta^2s>\eta\theta+1-q$ 
    \begin{equation}\label{x_k_drugi}
        x_k=-\tfrac{1}{1-q}\left(\eta s q^k+\tfrac{\eta\theta+1-q}{\eta q^k}-(\theta+\eta s)\right),  
    \end{equation}
    where $k=0,1,\ldots \text{ such that } \eta\theta+1-q<s\eta^2q^{2k}$.
\end{itemize}
Unfortunately, despite being included in the statement of Theorem 1.2 and Corollary 1.3 in \cite{bib_BrycMatysiakWesolowski_1}, the construction of the bi-Poisson in case $1+\eta\theta=\max\{0,q\}$, which will be needed in the sequel, is missing. We will fill this gap in Theorem \ref{bip} below. As we will see below, this boundary case is quite different from the case when $1+\eta \theta>\max\{0,q\}$.
\begin{theorem}\label{bip}
Assume that
$$
1+\eta\theta=\max\{q,0\}.
$$
Then the bi-Poisson process $\mathrm{QH}(\eta,\theta;0,0;q)$ exists.
\end{theorem}
\begin{proof} We use the three-term recurrence \eqref{rekurencja_na_wiel_Q}, which remains valid in the case we consider. The proof splits into two cases: $1+\eta \theta=q>0$ and $1+\eta \theta=0\ge q$.
\begin{itemize}
\item If $1+\eta \theta=q>0$, we get that $\eta\ne0$ (since $|q|<1$). Then, the calculation shows that for $s>0$ 
$$Q_n(y;0,s,0)=\tfrac{1}{\beta^n}p_n(\beta y+\alpha),\quad n\geq0,$$
where $\alpha=\tfrac{1-q+\eta^2s}{1-q}$, $\beta=\eta$ and $(p_n)_{n\geq0}$ are normalized Al-Salam-Carlitz of the first type polynomials with parameter $a=-\tfrac{s\eta^2}{1-q}<0$, see Section \ref{Al-Salam-Carlitz} in the Appendix. As a result, we can read out that the orthogonality measure $\mathds{Q}_{0,s}(0,\mathrm{d}y)$ for polynomials  $\{Q_n(y;0,s,0):n\geq0\}$ exists and is purely atomic with atoms 
\begin{equation}\label{x_k_jeden_prim}
x_k=-\tfrac{1}{1-q}\left(\theta q^k-(\theta+\eta s)\right), \quad k=0,1,2,\ldots,
\end{equation}
and
\begin{equation}\label{x_k_drugi_prim}
x_k=-\tfrac{1}{1-q}\left(\eta s q^k-(\theta+\eta s)\right), \quad k=0,1,2,\ldots,
\end{equation}
which coincide with \eqref{x_k_pierwszy} and \eqref{x_k_drugi} in case when $1+\eta\theta=q$.  Since $\mathcal{C}_n(\cdot,s)$ is a continuous function, we get that $x_k$ given by \eqref{x_k_jeden_prim} or \eqref{x_k_drugi_prim} are in $\mathcal{U}_s$. As a result, the analog of Lemma $2.4.$ in \cite{bib_BrycMatysiakWesolowski_1} holds for $1+\eta\theta=q>0$. Thus, there exists an appropriate quadratic harness since all other elements of the proof in \cite{bib_BrycMatysiakWesolowski_1}  remain unchanged.\\
Moreover, for $t>s>0$ and $x_k$ given by \eqref{x_k_jeden_prim} we have 
$$Q_n(y;x_k,t,s)=\tfrac{1}{\beta^n}p_n(\beta y+\alpha),\quad n\geq0,$$
where $\alpha=-\tfrac{\theta+\eta t}{1-q}$, $\beta=\tfrac{s\eta(1-q)}{t((\eta x_k+1)(1-q)-s\eta^2)}$ and $(p_n)_{n\geq0}$ are the normalized Big $q$-Jacobi polynomials with parameters $a=\tfrac{s}{tq}$, $b=0$ and $c=-\tfrac{s\eta^2}{q((\eta x_k+1)(1-q)-s\eta^2)}=-\tfrac{s\eta^2}{(1-q)q^{k+1}}<0$, see Section \ref{q-Racah polynomials}. In this case, we can easily read out the formula for $\mathds{Q}_{s,t}(x_k,\mathrm{d}y)$, which is also purely atomic.\\
For $t>s>0$ and $x_k$ given by \eqref{x_k_drugi_prim} we have that
$$\mathcal{C}_n(x_k,s)=\tfrac{\eta^2s}{1-q}q^{2(n-1)}(1-q^{k-n+1}), \quad n\geq1.$$
Then for $n<k+1$ and $n=k+1$ we have $\mathcal{C}_n(x_k,s)>0$ and $\mathcal{C}_n(x_k,s)=0$, respectively. Consequently, $\mathds{Q}_{s,t}(x_k,\mathrm{d}y)$ is a probability measure supported on distinct zeros of the polynomial $Q_{k+1}(y;x_k,t,s)$, see \cite[Theorem A.1.]{bib_BrycWesolowski_2010}.
\item If $1+\eta \theta=0\geq q$, then $\mathcal{B}_1(0,s,0)=t$ and $\mathcal{B}_2(0,s,0)=0$, $s>0$, so according to \cite[Theorem A.1.]{bib_BrycWesolowski_2010}, the probabilistic orthogonality measure $\mathds{Q}_{0,s}(0,\mathrm{d}y)$ is supported on two points 
$$x_{1,2}=\tfrac{\eta s+\theta\pm\sqrt{(\eta s+\theta)^2+4s}}{2},$$
which are the zeros of the polynomial $Q_{2}(y;0,s,0)$. Then, 
$$\mathcal{C}_{1}(x_{1,2},s)=\tfrac{\eta^2s+1\pm |\eta^2s+1|}{2}$$
and
$$\mathcal{C}_{2}(x_{1,2},s)=q\tfrac{-1-\eta^2s\pm|\eta s+1|}{2}.$$
As a result, both points are in $\mathcal{U}_s$ and the analog of Lemma 2.4. in \cite{bib_BrycMatysiakWesolowski_1} holds for $1+\eta\theta=0\geq q$, and the quadratic harness with proper parameters exists.\\
Furthermore, for $0<s<t$ we have $\mathcal{C}_{1}(x_{2},s)=0$, so $\mathcal{B}_{1}(x_{2},t,s)=0$ and \cite[Theorem A.1.]{bib_BrycWesolowski_2010} implies that $\mathds{Q}_{s,t}(x_2,\mathrm{d}y)$ is supported only on $x_2$, so $x_2$ is an absorbing state.\\
Since $\mathcal{C}_{1}(x_{1},s)>0$ and $\mathcal{C}_{2}(x_{1},s)=0$, hence for $0<s<t$ we obtain that $\mathcal{B}_{1}(x_{1},t,s)>0$ and $\mathcal{B}_{2}(x_{1},t,s)=0$. From \cite[Theorem A.1.]{bib_BrycWesolowski_2010} we conclude that the measure $\mathds{Q}_{s,t}(x_1,\mathrm{d}y)$ is supported on two distinct points, the zeros of $Q_2(y;x_1,t,s)$.
\end{itemize}
\end{proof}
\subsection{Quadratic harness $QH(\eta;\theta;0,\tau;q)$}\label{rozdzial_QH_z_tau}
The construction of $QH(\eta;\theta;0,\tau;q)$ was done in \cite[Theorem 1.1]{bib_BrycWesolowski_2010} when $BC$, $BD$, $qBC$ and $qBD$ are in $\mathds{C}\setminus[1,\infty)$, where
\begin{equation}\label{ABCD}
    A=0, \quad B=-\tfrac{\eta}{\sqrt{1-q+\eta\widetilde{\theta}}},\quad C=-\tfrac{\widetilde{\theta}+\tfrac{\eta\tau}{1-q}-\sqrt{\theta^2-4\tau}}{2\sqrt{1-q+\eta\widetilde{\theta}}} \quad \text{and}\quad D=-\tfrac{\widetilde{\theta}+\tfrac{\eta\tau}{1-q}+\sqrt{\theta^2-4\tau}}{2\sqrt{1-q+\eta\widetilde{\theta}}}
\end{equation}
and $\widetilde{\theta}$ is given in \eqref{theta_falka}.\\
In this case, the quadratic harness $(X_t)_{t\geq0}$ is a linear transformation of an Askey-Wilson process $(Y_t)_{t\geq0}$ with the parameters \eqref{ABCD}, i.e.,
\begin{equation}\label{X a Y}
    X_t:=\tfrac{2\sqrt{T(t)}Y_{T(t)}-Bt-C-D}{\sqrt{(1-q)(1-BC)(1-BD)}},\qquad t\geq0
\end{equation}
and $T(t)=t+CD$, see \cite[(2.28)]{bib_BrycWesolowski_2010}.
Since polynomials $(\overline{w}_n(y;a,b,c,d))_{n\geq0}$ satisfying \cite[(3.12)]{bib_BrycWesolowski_2010} are orthogonal to uniquely determined transition probabilities of $(Y_t)_{t\geq0}$, we can easily read out polynomials $(\widetilde{Q}_n(y;x,t,s))_{n\geq0}$ which are orthogonal to the transition probabilities of $(X_t)_{t\geq0}$. Indeed, formula \eqref{X a Y} implies that for $n\geq0$
\begin{equation}\label{jak_otrzymac_wielomiany}
    \widetilde{Q}_{n}(y;x,t,s)=\tfrac{1}{u^n}\overline{w}_n\left(uy+w;a,b,c,d\right),
\end{equation}
where $t\geq s\geq0$, $u=\tfrac{2\sqrt{T(t)}}{\sqrt{(1-q)(1-BC)(1-BD)}}=\tfrac{2\sqrt{\tfrac{\alpha_t}{1-q}}\sqrt{\eta\widetilde{\theta} +1-q}}{1-q}$, $w=-\tfrac{Bt+C+D}{\sqrt{(1-q)(1-BC)(1-BD)}}=\tfrac{\widetilde{\theta}+\tfrac{\eta\alpha_t}{1-q}}{1-q}$, and parameters $a$, $b$, $c$ and $d$ satisfy
\begin{equation}\label{parametry_dla_wiel_martyngalowych}
    a=A\sqrt{T(t)}, \quad b=B\sqrt{T(t)}, \quad c=C/\sqrt{T(t)}, \quad d=D/\sqrt{T(t)}
\end{equation}
when $s=0$ and $x=0$
or
$$a=A\sqrt{T(t)}, \quad b=B\sqrt{T(t)},\quad c=(z_x+\sqrt{1-z_x^2})\sqrt{\tfrac{T(s)}{T(t)}}, \quad d=(z_x+\sqrt{1-z_x^2})\sqrt{\tfrac{T(s)}{T(t)}},$$
when $s>0$ and $z_x=\tfrac{\sqrt{(1-q)(1-BC)(1-BD)}}{2T(s)}x+\tfrac{Bt+C+D}{2T(s)}$. Thus, polynomials $(\widetilde{Q}_n(y;x,t,s))_{n\geq0}$ satisfy the following three-step recurrence:
\begin{equation}
\label{rekurencja_na_wiel_Q_falka}
\begin{split}
y\widetilde{Q}_n(y;x,t,s)=\widetilde{Q}_{n+1}(y;x,t,s)+\mathcal{\widetilde{A}}_{n}(x,&t,s)\widetilde{Q}_n(y;x,t,s)\\
&+\mathcal{\widetilde{B}}_{n}(x,t,s)\widetilde{Q}_{n-1}(y;x,t,s), \qquad n \geq0,
\end{split}
\end{equation}
with $\widetilde{Q}_{-1}(y;x,t,s)=0$, $\widetilde{Q}_{0}(y;x,t,s)=1$ and the coefficients are given by
\begin{equation*}
\begin{split}
\mathcal{\widetilde{A}}_{n}(x,t,s)&=q^nx+[n]_q(\eta t +\theta+\eta\tau([n]_q+[n-1]_q)-(1+q)q^{n-1}s\eta),\\
\mathcal{\widetilde{B}}_{n}(x,t,s)&=[n]_q(t-sq^{n-1}+\tau[n-1]_q)\{1+\eta x q^{n-1}+\eta[n-1]_q(\theta+\eta\tau[n-1]_q-s\eta q^{n-1})\}.
\end{split}
\end{equation*}	
For $n=0$ these formulas should be interpreted as $\mathcal{\widetilde{A}}_{0}(x,t,s)=x$ and $\mathcal{\widetilde{B}}_{0}(x,t,s)=0$.
Especially $P_n(y;t):=\widetilde{Q}_n(y;0,t,0)$, $n\geq0$, are martingale polynomials for $QH(\theta,\eta;0,\tau;q)$, compare with  \cite[Theorem 4.5]{bib_BrycMatysiakWesolowski} and they satisfy:
\begin{equation*}
   \begin{split}
        yP_n(y; t)= P_{n+1}(y; t)  &+ (\eta t + \theta + ([n]_q + [n-1]_q )\eta\tau) [n]_q P_n(y; t)\\
        &+ (t + \tau [n-1]_q ) (1 + [n-1]_q \eta\theta + [n-1]^2_q \eta^2\tau)[n]_q P_{n-1}(y; t), \quad n \geq 0,
    \end{split}
\end{equation*}
with $P_{-1}(y;t)=0$ and $P_0(y;t)=1$. Moreover, since they come from the Askey-Wilson polynomials (see \eqref{jak_otrzymac_wielomiany}) with parameters \eqref{parametry_dla_wiel_martyngalowych}, we can read out the support of the orthogonal measures for $(P_n(y;t))_{n\geq0}$:
\begin{itemize}
   \item When $\theta^2< 4\tau$, then $C$ and $D$ are complex conjugates, so the process exists and 
    $$|c|^2=|d|^2=cd=\tfrac{\tau}{\alpha_t}<1,$$
    Furthermore, $|b|< 1$ if and only if $\eta^2t< 1-q+\eta\theta$.\\
    Therefore  \cite[Lemma 3.1]{bib_BrycWesolowski_2010} with linear transformation \eqref{jak_otrzymac_wielomiany} imply that for $\eta^2t< 1-q+\eta\theta$ the orthogonal measure  for after  $\{P_n(y;t):n\geq0\}$ has only continuous part and the support of $QH(\eta,\theta;0,\tau;q)$  at time $t\geq0$ satisfying $\eta^2t\leq 1-q+\eta\theta$ is
    \small
    \begin{equation}\label{przedzial_uw}
    [-u+w,u+w]=\left[\tfrac{\widetilde{\theta}+\eta \tfrac{\alpha_t}{1-q}-2\sqrt{\tfrac{\alpha_t}{1-q}}\sqrt{\eta\widetilde{\theta}+1-q}}{1-q},\tfrac{\widetilde{\theta}+\eta \tfrac{\alpha_t}{1-q}+2\sqrt{\tfrac{\alpha_t}{1-q}}\sqrt{\eta\widetilde{\theta}+1-q}}{1-q}\right].
   \end{equation}
   \normalsize
    For $\eta^2t\geq 1-q+\eta\theta$, in addition to a continuous part, we also have in the support discrete points, which are given by 
    $$\tfrac{u}{2}(b q^k+\tfrac{1}{b}q^{-k})+w$$
    for $k=0,1,2,\ldots$ satisfying $\eta^2\alpha_tq^{2k}>(1-q)(1-q+\eta\widetilde{\theta})$.
    \item  When $\theta^2\geq 4\tau$, then $B$, $C$ and $D$ are real. Moreover, if all these parameters are less than $1$, then $QH(\eta,\theta;0,\tau;q)$ exists. If none of $BC$, $BD$, $qBC$ and $qBD$ is at least $1$, then process exists and
    $$cd=\tfrac{\tau}{\alpha_t}<1,\quad  bd=BD<1,\quad  cd=CD<1,$$
    $$qcd=q\tfrac{\tau}{\alpha_t}<1,\quad  qbd=qBD<1,\quad  qcd=qCD<1.$$
    According to \cite[Lemma 3.1]{bib_BrycWesolowski_2010}, we have $m_1=m_2=0$ and orthogonal measure for $(P_n(y;t))_{n\geq0}$ has in this case only continuous part which is supported on \eqref{przedzial_uw}.
\end{itemize}

\section*{Acknowledgements}
This research was supported by grant Beyond POB II no. 1820/366/Z01/2021 within the Excellence Initiative: Research University (IDUB) programme of the Warsaw Univ. of Technology, Poland.

\end{document}